\documentclass{amsart}
\usepackage{enumitem}
\usepackage[colorlinks]{hyperref}
\usepackage{amssymb}
\usepackage{tikz-cd}

\newtheorem{lemma}{Lemma}
\newtheorem{proposition}{Proposition}
\newtheorem{theorem}{Theorem}
\newtheorem*{theorem*}{Theorem}
\newtheorem{corollary}{Corollary}
\newtheorem*{corollary*}{Corollary}
\newtheorem*{proposition*}{Proposition}
\newtheorem{definition}{Definition}

\numberwithin{equation}{section}

\theoremstyle{remark}
\newtheorem{remark}{Remark}

\DeclareMathOperator{\Div}{div}

\DeclareMathOperator{\Spec}{Spec}

\DeclareMathOperator{\Pic}{Pic}
\DeclareMathOperator{\Picsch}{\mathbf{Pic}}

\DeclareMathOperator{\gp}{gp}
\DeclareMathOperator{\gr}{gr}
\DeclareMathOperator{\rk}{rk}

\DeclareMathOperator{\codim}{codim}
\DeclareMathOperator{\red}{red}
\DeclareMathOperator{\cha}{char}

\DeclareMathOperator{\cok}{cok}

\DeclareMathOperator{\Gal}{Gal}

\DeclareMathOperator{\dlog}{dlog}

\DeclareMathOperator{\pr}{pr}
\DeclareMathOperator{\sh}{sh}

\DeclareMathOperator{\et}{\acute{e}t}

\DeclareMathOperator{\depth}{depth}
\DeclareMathOperator{\D}{D}

\title[Cohomological flatness over discrete valuation rings]{Cohomological flatness over discrete valuation rings: numerical and logarithmic criteria}
\author[O. Gabber]{Ofer Gabber}
\email{gabber@ihes.fr}
\author[R. Lodh]{R\'emi Lodh}
\email{remi.shankar@gmail.com}

\begin{document}
\begin{abstract}
We give sufficient conditions for cohomological flatness (in dimension 0) over discrete valuation rings, generalizing classical results of Raynaud in two different ways. The first is a higher dimensional generalization of Raynaud's numerical criteria, in both the variant for the multiplicity of the special fibre and that for the index of the generic fibre. The second is a logarithmic criterion: we show that, over a log regular base, a proper flat fs log smooth morphism is cohomologically flat in dimension 0. We apply this latter result to curves and torsors under abelian varieties with good reduction, providing necessary and sufficient conditions for the log smoothness of their regular models over arbitrary discrete valuation rings.
\end{abstract}
\maketitle

\section{Introduction}
As part of his groundbreaking work on algebraic spaces, Michael Artin proved a theorem that effectively settled the question of the existence of the Picard scheme.

\begin{theorem}[Artin \cite{artinformalI,artinversal}]\label{artin}
Assume $f:X\to T$ is a finitely presented morphism of schemes that is proper and flat. If $f$ is cohomologically flat in dimension 0, then the Picard functor $\Picsch_{X/T}$ is representable by an algebraic space.
\end{theorem}

This fundamental result was the crowning achievement of the work of a number of mathematicians on the Picard functor; see \cite{picard} for the history of this problem. Notably, Grothendieck had studied the problem in detail using the arsenal of new ideas and techniques he was introducing into algebraic geometry, reporting on his insights in the S\'eminaire Bourbaki \cite{grothpicard}. In particular, he had identified the condition of cohomological flatness in dimension 0 as one of central importance to the problem and suggested that it might be sufficient to ensure the representability of the Picard functor as a scheme, cf. \cite[V, 5.2]{grothpicard}.

As Artin's theorem suggests, this turned out to be too optimistic, and Mumford soon furnished a counterexample \cite[VI, \S0]{grothpicard}.\footnote{Mumford's example seems to have been an important step towards Theorem \ref{artin}, cf. \cite[p. 210]{neronmodels}} But the importance of cohomological flatness was noted by Michel Raynaud, who studied it in relation to questions of representability of functors related to $\Picsch$ in the case $T$ is the spectrum of a discrete valuation ring. The fruits of his efforts were published in \cite{raynaud}, which, among other results, established the following `numerical criterion' for cohomological flatness.

\begin{theorem}[Raynaud \cite{raynaud}]\label{raynaudintro}
Let $T$ be the spectrum of a discrete valuation ring and $f:X\to T$ a proper morphism. Assume the following conditions hold:
\begin{enumerate}[label=\normalfont(\alph*)]
\item $X$ is integrally closed in the generic fibre of $f$
\item $f_*\mathcal{O}_X=\mathcal{O}_T$
\item either $\dim X\leq 2$ or $T$ has residue characteristic 0.
\end{enumerate}
If the gcd of the geometric multiplicities of the components of the special fibre of $f$ is invertible on $T$, then $f$ is cohomologically flat in dimension 0.
\end{theorem}

Raynaud gave examples in \cite{raynaud} that show that the condition on the gcd of the multiplicities is necessary in general. This result raises two immediate questions:
\begin{enumerate}[label=\normalfont(\arabic*)]
\item Does Theorem \ref{raynaudintro} hold without restriction on $\dim X$?
\item If the gcd of the multiplicities is not invertible on $T$, is there a condition of local nature that one can substitute for (a) to guarantee cohomological flatness in dimension 0?
\end{enumerate}

In this paper we answer both of these questions in the affirmative. Firstly, the answer to (1) is a straightforward `yes', as will follow from a closely related result that we now describe.

In \cite{raynaud}, Raynaud gave another numerical criterion for cohomological flatness in dimension 0, using the index instead of the gcd of the multiplicities. Recall that the \emph{index} of a scheme of finite type over a field is the gcd of the degrees of its closed points. For a relative curve over a strictly henselian discrete valuation ring, Raynaud showed that $f$ is cohomologically flat in dimension 0 if the index of the generic fibre of $f$ is 1 (\cite[8.2.1]{raynaud}). We prove that it is enough to assume that the index is prime to the residue characteristic and, moreover, that this holds in arbitrary relative dimension.

\begin{theorem}\label{gabberraynaudindex}
Let $T$ be the spectrum of a discrete valuation ring and $f:X\to T$ a proper morphism. Assume the following conditions hold:
\begin{enumerate}[label=\normalfont(\alph*)]
\item $X$ is integrally closed in the generic fibre of $f$
\item the special fibre of $f$ is connected.
\end{enumerate}
If the index of the generic fibre of $f$ is invertible on $T$, then $f$ is cohomologically flat in dimension 0.
\end{theorem}

This result implies Raynaud's Theorem \ref{raynaudintro}, without limitation on the dimension, thereby providing a positive answer to question (1). We also show a variant where the index is replaced by the gcd of the \emph{apparent} multiplicities of the special fibre (see Theorem \ref{gabberraynaud}).

For question (2), Raynaud's examples show that the question is not as straightforward. There was essentially no progress until quite recently, when we realized that cohomological flatness is a necessary condition for the log smoothness of regular curves over discrete valuation rings \cite{logsmoothcurves} (see also \cite{mitsuismeets}). This recent result for curves begged the question of whether it held in greater generality, and indeed this is the case.

\begin{theorem}\label{main}
Let $T$ be a locally noetherian scheme, $M_T$ an fs log structure on $T$, and $f_{\log}:(X,M_X)\to (T,M_T)$ a morphism of fs log schemes whose underlying morphism $f:X\to T$ is proper and flat. If $(T,M_T)$ is log regular and $f_{\log}$ is log smooth, then $f$ is cohomologically flat in dimension 0.
\end{theorem}

The proof of Theorem \ref{main} is surprisingly simple, the idea being to pass to a finite extension of $T$ where $f_{\log}$ acquires reduced fibres (which exists thanks to Tsuji \cite{tsuji}), and then to descend by taking the quotient under an action of roots of unity.

Theorem \ref{main} answers question (2), in the sense that if we fix a log structure $M_X$, then we have a local criterion for cohomological flatness in dimension 0 that applies to situations where the gcd of the multiplicities of the special fibre may be divisible by the residue characteristic. This, however, depends on the choice of $M_X$. Nevertheless, in many situations there is a natural choice of $M_X$, and the converse to Theorem \ref{main} may even be true. Namely, for torsors under abelian varieties with good reduction, Raynaud \cite{raynaudquotient} constructed canonical regular projective models over discrete valuation rings for which we show the following.

\begin{theorem}\label{raynaudmodelcriterion}
Assume $T$ is the spectrum of a discrete valuation ring and $M_T$ its canonical log structure. For $A/T$ an abelian scheme and $Y$ a torsor under the generic fibre of $A/T$, the following are equivalent:
\begin{enumerate}[label=\normalfont(\roman*)]
\item the Raynaud model of $Y$, with its canonical log structure, is log smooth over $(T,M_T)$
\item the Raynaud model of $Y$ is cohomologically flat in dimension 0 over $T$.
\end{enumerate}
\end{theorem}

This generalizes a result for genus 1 curves shown in \cite{mitsuismeets} and \cite{logsmoothcurves}. For curves of higher genus, regular normal crossings models have natural log structures, but the resulting morphism of log schemes might not be log smooth, even if the underlying scheme is cohomologically flat over $T$. So we provide necessary and sufficient conditions for the log smoothness of such models, extending the main result of \cite{logsmoothcurves}, notably to the case of an imperfect residue field---see Theorem \ref{curves}. This was an early motivation for Theorem \ref{main}.\\

To finish this introduction, we give a short overview of the paper and its main results.

We begin \S\ref{cohflatsection} by recalling the basic properties of cohomological flatness in dimension 0, before proving the converse to Artin's Theorem \ref{artin} (Proposition \ref{artinconverse}).

In \S\ref{numericalsection}, we recall the definition of multiplicities, Raynaud's condition (N), and show some basic properties. Then, we prove Theorem \ref{gabberraynaudindex} by studying divisiorial cycles with rational coefficients and their associated sheaves, the key result being Proposition \ref{vanishingcriterion}. We also prove a variant where the index of the generic fibre is replaced by the gcd of the multiplicities of the components of the special fibre (Theorem \ref{gabberraynaud}).

In \S\ref{logsection}, after some preliminaries we give the proof of Theorem \ref{main}, as described above. The rest of the section is devoted to applications to torsors under abelian schemes and curves.

In \S\ref{torsorsection}, we study models of torsors under the generic fibre of an abelian scheme over the spectrum of a discrete valuation ring $T$. We show Theorem \ref{raynaudmodelcriterion} and deduce necessary and sufficient conditions for the existence of a proper log smooth $T$-model (Corollary \ref{torsorlogsmooth}).

Finally, in \S\ref{curvesection} we give necessary and sufficient conditions for log smoothness of proper regular curves over the spectrum of a discrete valuation ring $T$, extending the main theorem of \cite{logsmoothcurves} to arbitrary residue fields and allowing horizontal components in the log structure (Theorem \ref{curves}). To show this, we establish some properties of regular log smooth $T$-schemes of arbitrary dimension (Proposition \ref{ncd}).

\subsubsection*{Acknowledgements}
We warmly thank Luc Illusie for his interest and assistance.

\subsection{Conventions}\label{conventions}
If $T$ denotes the spectrum of a discrete valuation ring we will write $t\in T$ for its closed point with residue field $k(t)$ and $u\in T$ its generic point with residue field $k(u)$. We let $p$ denote the characteristic of $k(t)$. The \emph{canonical log structure} of $T$ is the direct image log structure of the trivial log structure on $u$.

Irreducible components are endowed with the reduced induced subscheme structure.

We only consider fs log schemes. In particular, fibre products are taken in the category of fs log schemes. Schemes can be viewed as fs log schemes with the trivial log structure; in that case we may simply omit the log structure from notation. For a log structure $M$ on a scheme $X$, we write $\overline{M}=M/\mathcal{O}_{X}^*$.

All monoids will be commutative with unity. For a monoid $P$ we write $P^{\gp}$ for its Grothendieck group.

For an abelian group $A$ and integer $n$, we write $A[n]\subset A$ for the subset of elements annhilated by $n$, and $A[\text{tor}]=\varinjlim_nA[n]$ for the torsion subgroup of $A$. If $A$ is finite we denote its order $|A|$.

\section{Cohomological flatness}\label{cohflatsection}
Let $T$ be a scheme and $f:X\to T$ a finitely presented morphism that is proper and flat.
 
\subsection{Definition and basic properties}
Following Raynaud \cite[1.4]{raynaud}, we will say that $f$ is  \emph{cohomologically flat} if $f$ is cohomologically flat in dimension 0 (\cite[III\textsubscript{2}, 7.8.1]{ega}).\footnote{Warning: this conflicts with the terminology of \cite[III\textsubscript{2}, 7.8.1]{ega}!} We gather some basic properties of cohomological flatness.

\begin{proposition}\label{basechange}
\begin{enumerate}[label=\normalfont(\roman*)]
\item $f$ is cohomologically flat if and only if the formation of $f_*\mathcal{O}_X$ commutes with base change.
\item $f$ is cohomologically flat if and only if there is a faithfully flat morphism $S\to T$ such that $f\times_TS$ is cohomologically flat.
\item If $g:T\to S$ is a finite locally free morphism and $f$ is cohomologically flat, then so is $g\circ f$.
\item If $f$ has geometrically reduced fibres, then $f$ is cohomologically flat.
\item $f$ is cohomologically flat if and only if, for all $t\in T$, the homomorphism $(f_*\mathcal{O}_X)_t\to H^0(X_t,\mathcal{O}_{X_t})$ is surjective.
\item If $T$ is reduced, then $f$ is cohomologically flat if and only if $\dim_{k(t)}H^0(X_t,\mathcal{O}_{X_t})$ is a locally constant function on $T$.
\item If $T$ is the spectrum of a discrete valuation ring, then $f$ is cohomologically flat if and only if $H^1(X,\mathcal{O}_X)$ is a torsion-free $\mathcal{O}(T)$-module.
\end{enumerate}
\end{proposition}
\begin{proof}
Under our hypotheses on $f$, (i) follows from the definition (cf. \cite[\S 8.1, p. 206]{neronmodels}) and implies (ii) via flat base change. For (iii), use (i) and affine base change. To check the remaining assertions, by (i) we may assume $T$ is affine and then, by a limit argument (cf. \cite[IV\textsubscript{3}]{ega}), further reduce to the case $T$ is noetherian. In this case, (iv) is shown in \cite[III\textsubscript{2}, 7.8.6]{ega}, (v) in \cite[III\textsubscript{2}, 7.7.10]{ega}, and (vi)--(vii) follow from (v).
\end{proof}

\subsection{Converse to Artin's theorem}
We show that the converse to Artin's Theorem \ref{artin} holds, providing a criterion for cohomological flatness (cf. \cite[\S 8.3, Rmk. 2]{neronmodels} for the case $T$ reduced). Recall that the relative Picard functor $\Picsch_{X/T}$ is the fppf sheafification of the functor on $T$-schemes $S\rightsquigarrow\Pic(X\times_TS)$. Since $X/T$ is proper, $\Picsch_{X/T}$ is also the \'etale sheafification of this functor (\cite[\S 8.1, p. 203]{neronmodels}).

\begin{proposition}\label{artinconverse}
If $\Picsch_{X/T}$ is representable by an algebraic space, then $f$ is cohomologically flat.
\end{proposition}
\begin{proof}
By Proposition \ref{basechange} (v), it suffices to show that the canonical map $H^0(X,\mathcal{O}_X)\to H^0(X_k,\mathcal{O}_{X_k})$ is surjective when $T=\Spec(A)$ is the spectrum of a local ring with residue field $k$. Moreover, by Proposition \ref{basechange} (ii), we may assume $A$ to be strictly henselian. Define
\[ C:=\cok(H^0(X,\mathcal{O}_{X}^*)\to H^0(X_k,\mathcal{O}_{X_k}^*)) \]

\begin{lemma}
If $C=0$, then $H^0(X,\mathcal{O}_X)\to H^0(X_k,\mathcal{O}_{X_k})$ is surjective.
\end{lemma}
\begin{proof}
We first show that $H^0(X,\mathcal{O}_X)\to H^0(X_k,\mathcal{O}_{X_k})$ induces a bijection on the finite sets of idempotents of each ring. For $A$ noetherian, this follows from Zariski's connectedness theorem (\cite[III\textsubscript{1}, 4.3.3]{ega}), since $A$ is henselian. In general, write $A$ as a direct limit of local henselian noetherian subalgebras. By \cite[IV\textsubscript{3}, 8.9.1]{ega}, we can assume that $X$ is the base change of a proper scheme over such a subalgebra, and then deduce the statement by taking the limit.

Now, write $H^0(X_k,\mathcal{O}_{X_k})=\oplus_{i=1}^nR_i$ with $R_i$ finite local $k$-algebras. Since $k$ is infinite, given $(a_1,...,a_n)\in H^0(X_k,\mathcal{O}_{X_k})$ there are $b_1,..,b_n\in A^*$ such that the residue classes $\bar{a}_i\neq\bar{b}_i$ in the reduction of $R_i$ for $1\leq i\leq n$. Then $(a_1-\bar{b}_1,...,a_n-\bar{b}_n)$ is a unit, hence, since $C=0$, lifts to $(c_1,...,c_n)\in H^0(X,\mathcal{O}_X^*)$ and $(c_1+b_1,...,c_n+b_n)$ provides a lift of $(a_1,...,a_n)$.
\end{proof}

So it suffices to show $C=0$. Consider the fibre product $B:=A\times_kA$. Write $P=\Picsch_{X/T}$ and consider the diagonal map $P(B)\to P(A)\times P(A)$.

\begin{lemma}
$\ker(P(B)\to P(A)\times P(A))\simeq C$
\end{lemma}
\begin{proof}
Consider the category $\mathcal{C}$ of triples $(\mathcal{E}_1,\mathcal{E}_2,s)$, where $\mathcal{E}_i$ ($i=1,2$) are vector bundles on $X$ and $s:\mathcal{E}_1\otimes_Ak\to\mathcal{E}_2\otimes_Ak$ is an isomorphism. There is a natural functor from the category of vector bundles on $X_B$ to $\mathcal{C}$ and it is an equivalence by \cite[2.2]{ferrand}, with an inverse given by
\[ \mathcal{E}=\{(e_1,e_2)\in\mathcal{E}_1\times\mathcal{E}_2:s(\overline{e}_1)=\overline{e}_2\} \]
where, for $e_i\in\mathcal{E}_i$, $\overline{e}_i\in\mathcal{E}_i\otimes_Ak$ is its image. It follows that a line bundle $\mathcal{E}$ on $X_B$ with trivial restrictions $\mathcal{E}_i\simeq\mathcal{O}_X$ is trivial if and only if $s\in\text{Isom}(\mathcal{E}_1\otimes_Ak,\mathcal{E}_2\otimes_Ak)\simeq H^0(X_k,\mathcal{O}_{X_k}^*)$ lifts to $H^0(X,\mathcal{O}_X^*)$. This implies the assertion.
\end{proof}

To complete the proof of Proposition \ref{artinconverse}, note that since $P$ is representable we have $P(B)=P(A)\times_{P(k)}P(A)$ (\cite[II]{grothpicard}). Thus, the lemma implies $C=0$.
\end{proof}

\section{Numerical criteria}\label{numericalsection}
$f:X\to T$ a morphism of noetherian schemes. In this section, $T$ will be the spectrum of a discrete valuation ring and we will use the notation of \S\ref{conventions}. 

\subsection{Multiplicities}\label{multiplicity}
Assume $f$ is of finite type with $X_t\neq\emptyset$ and let $\eta\in X_t$ be a maximal point. The \emph{(apparent) multiplicity $d_{\eta}$ of $\eta$ in $X_t$} is, by definition, the length of the artinian local ring $\mathcal{O}_{X_t,\eta}$.

Let $\bar{t}:=\Spec(\overline{k(t)})$ be the spectrum of an algebraic closure of $k(t)$ and $\bar{\eta}\in X_t\times_t\bar{t}$ a point lying over $\eta$. The \emph{geometric multiplicity $\delta_\eta$ of $\eta$ in $X_t$} is defined $\delta_{\eta}:=d_{\bar{\eta}}$ \cite[\S 9.1, Def. 3]{neronmodels}. By \cite[\S 9.1, Lemma 4]{neronmodels}, it is equal to the \emph{geometric length} of $\mathcal{O}_{X_t}$ at $\eta$, as defined in \cite[IV\textsubscript{2}, 4.7.5]{ega}.

The \emph{(apparent) multiplicity} $d_f$ (resp. \emph{geometric multiplicity} $\delta_f$) of $X_t$ is the gcd of the integers $d_\eta$ (resp. $\delta_{\eta}$) as $\eta$ ranges over the finite set of maximal points of $X_t$.

\begin{lemma}\label{coprime}
\begin{enumerate}[label=\normalfont(\roman*)]
\item $\delta_\eta$ is independent of the choice of algebraic closure of $k(t)$ and $\bar{\eta}$.
\item $(\delta_f,p)=1$ if and only if there is an irreducible component of $X_t$ whose apparent multiplicity is prime to $p$ and whose reduction is geometrically reduced.
\item If $f$ is flat and $X$ is regular at a maximal point $\eta\in X_t$, then $d_{\eta}$ is equal to the valuation at $\eta$ of a uniformizer of $\mathcal{O}(T)$. 
\item Assume $f$ is flat and $X$ is regular at the maximal points of its special fibre. If $g:T\to S$ is a finite flat morphism of spectra of discrete valuation rings, then $d_{g\circ f}=ed_f$ and $\delta_{g\circ f}=e\delta_f$, where $e$ is the ramification index of $g$.
\item Let $h:X'\to X$ be a proper birational morphism of flat $T$-schemes which are regular at the maximal points of their special fibres. Then $\delta_{f\circ h}|\delta_f$.
\item If $S\to T$ is a morphism of spectra of discrete valuation rings such that $S_t$ is the spectrum of a separable extension of $k(t)$, then $d_f=d_{f\times_TS}$ and $\delta_f=\delta_{f\times_TS}$.
\end{enumerate}
\end{lemma}
\begin{proof}
(i) follows from \cite[IV\textsubscript{2}, 4.7.8]{ega}, (ii) from \cite[\S 9.1, Lemma 4]{neronmodels} and \cite[IV\textsubscript{2} 4.7.10]{ega}, and (iii) is a straighforward check.

For (iv), let $s\in S$ denote the closed point. Note that $X_t\to X_s$ is a closed immersion and a homeomorphism. If $\eta\in X_t$ is a maximal point and $\zeta\in X_s$ its image, then $\mathcal{O}_{X,\eta}=\mathcal{O}_{X,\zeta}$. Computing valuations and using (iii), one sees that $d_\zeta=ed_\eta$, where $e$ is the ramification index of $T\to S$. Since $X_t$ and $X_s$ have the same reduction, by \cite[IV\textsubscript{2}, 4.7.8]{ega} we have $\delta_{\zeta}/d_{\zeta}=\delta_{\eta}/d_{\eta}$. Thus, $\delta_{\zeta}=e\delta_{\eta}$, hence $\delta_{g\circ f}=e\delta_f$ and the claim follows.

Finally, (v) holds because, by the valuative criterion for properness, every maximal point of $X_t$ lifts to a unique maximal point of $X'_t$, and (vi) follows from the definitions.
\end{proof}

\subsection{Raynaud's condition (N)}
Raynaud stated Theorem \ref{raynaudintro} under a condition he called (N). Recall that this condition states:
\begin{equation}\label{N}
\text{$X_t$ is $S_1$ and $X$ is regular at the maximal points of $X_t$}\tag{N}
\end{equation}
For instance, this holds if $X$ is normal. The next result, due to Koll\'ar, provides a useful characterization of condition (N). For its statement, let $u\in T$ be the generic point and $j:X_u\to X$ be the canonical map. We say that \emph{$X$ integrally closed in $X_u$} if $X$ is integrally closed in $j_*\mathcal{O}_{X_u}$ (in the sense of \cite[II, 6.3]{ega}).

\begin{proposition}[\cite{kollar}]\label{kollar}
The following are equivalent
\begin{enumerate}[label=\normalfont(\roman*)]
\item $f$ is flat and satisfies (N)
\item $X$ is integrally closed in $X_u$
\item $X_t\subset X$ is a normal pair (cf. \cite{kollar}).
\end{enumerate}
\end{proposition}
\begin{proof}
Condition (ii) means that $X$ is the \emph{relative normalization} of the pair $X_t\subset X$ in the sense of \cite[Def. 1]{kollar}, which is equivalent to (iii) by definition (cf. \cite[Def. 2]{kollar}). The equivalence of (i) and (iii) follows easily from \cite[Cor. 7]{kollar}.
\end{proof}

We gather some facts about (N) in the next couple of lemmata.

\begin{lemma}\label{complete}
Assume $f:X\to T$ is a morphism of finite type. Let $S\to T$ be a morphism of spectra of discrete valuation rings such that $S_t$ is the spectrum of a separable extension of $k(t)$. Then $f$ satisfies (N) if and only if $f\times_TS$ does.
\end{lemma}
\begin{proof}
See \cite[6.1.7]{raynaud}.
\end{proof}

\begin{lemma}\label{steinreduction}
Assume $f$ is proper, flat, and satisfies (N). Let $X\overset{f'}{\to} T'\to T$ be the Stein factorization of $f$. Then
\begin{enumerate}[label=\normalfont(\roman*)]
\item $T'$ is the spectrum of a one-dimensional semilocal normal ring
\item $f'$ is proper and flat
\item for each closed point $t'\in T'$, $f'\times_{T'}\Spec(\mathcal{O}_{T',t'})$ satisfies (N)
\item $f$ is cohomologically flat if and only if $f'$ is.
\end{enumerate}
\end{lemma}
\begin{proof}
For (i), see \cite[6.1.8]{raynaud} (and use that $f$ is flat). Note that it implies (ii). For (iii), first note that since $f$ is flat and $X_t$ is $S_1$, the inequality $\depth\mathcal{O}_{X,x}\geq\min\{2,\dim\mathcal{O}_{X,x}\}$ is satisfied for all $x\in X_t$. Since $f'$ is flat, it follows from this that its closed fibres are $S_1$. Moreover, $X$ is clearly regular at the maximal points of the closed fibres of $f'$ (since this holds for $f$), so (iii) follows. Finally, (iv) follows from Proposition \ref{basechange} (iii).
\end{proof}

\subsection{Divisorial cycles}
From now on we assume $f$ is flat and satisfies (N). 

If $X_t\neq\emptyset$, let $X_1,...,X_c$ denote its irreducible components and, for $1\leq i\leq c$, $\eta_i\in X_i$ the generic point. Then, $\mathcal{O}_{X,\eta_i}$ is regular of dimension 1; let $v_i$ be the discrete valuation of $\mathcal{O}_{X,\eta_i}$, with the usual convention $v_i(0)=+\infty$. Let $\mathbb{N}_{\infty}=\mathbb{N}\cup\{+\infty\}$ and $\mathbb{Z}_{\infty}=\mathbb{Z}\cup\{+\infty\}$, with the obvious additive monoid structures. Denoting by $j:X_u\to X$ be the inclusion of the generic fibre of $f$, $v_i$ extends to a morphism of Zariski sheaves $\tilde{v}_i:j_*\mathcal{O}_{X_u}\to\eta_{i,*}\mathbb{Z}_{\infty}$ whose stalk at $\eta_i$ is $v_i$.

Define a map of (Zariski) sheaves
\begin{align*}
\Div_X:j_*\mathcal{O}_{X_u} &\to \oplus_{1\leq i\leq c}\eta_{i,*}\mathbb{Z}_{\infty} \\
s &\mapsto (\tilde{v}_i(s))_{1\leq i\leq c}
\end{align*}
If $X_t=\emptyset$, then we let $\Div_X$ be the zero map.

\begin{lemma}\label{integrallyclosed}
For $U\subset X$ open, we have $\mathcal{O}_X(U)=\{s\in j_*\mathcal{O}_{X_u}(U):\Div_U(s)\geq 0\}$.
\end{lemma}
\begin{proof}
Let $\mathcal{M}$ be the presheaf $\mathcal{M}(U)=\{s\in j_*\mathcal{O}_{X_u}(U):\Div_U(s)\geq 0\}$. Then
\[ \mathcal{M}=(j_*\mathcal{O}_{X_u})\times_{(\oplus_i\eta_{i,*}\mathbb{Z}_{\infty})}(\oplus_i\eta_{i,*}\mathbb{N}_{\infty}) \]
so $\mathcal{M}$ is a sheaf and clearly $\mathcal{O}_X\subset\mathcal{M}$. To prove that this inclusion is an isomorphism we may assume $X$ is local, in particular, finite dimensional. For finite dimensional $X$, we prove the assertion by induction on $\dim X$, the case $\dim X=1$ being straightforward. So assume the result for flat noetherian $T$-schemes of dimension less than $\dim X\geq 2$ satisfying (N).

Let $x\in X$ be the closed point and $U=X\setminus\{x\}$. Then $\dim U<\dim X$, so by induction the result holds for $U$. By (N) and flatness over $T$, we have $\depth\mathcal{O}_{X,x}\geq 2$, hence $\mathcal{O}_X(U)=\mathcal{O}_{X}(X)$. Since $U$ contains all of the maximal points of $X_t$ and $U_u=X_u$, we have $\mathcal{M}(U)=\mathcal{M}(X)$ and the claim follows.
\end{proof}

An (integral) \emph{divisorial cycle supported on $X_t$} is, by definition, a global section of the sheaf $\oplus_{1\leq i\leq c}\eta_{i,*}\mathbb{Z}$. We will often denote by $[X_i]$ the element $1\in H^0(X,\eta_{i,*}\mathbb{Z})$ and write divisiorial cycles as formal sums $\sum_{i=1}^ca_i[X_i]$. The \emph{class group} of such cycles is defined
\[ C(X):=\cok\left(H^0(X_u,\mathcal{O}_{X_u}^*)
\overset{\Div_X}{\to}H^0(X,\oplus_{1\leq i\leq c}\eta_{i,*}\mathbb{Z})\right) \]
Define
\[ d'_f:=|C(X)[\text{tor}]| \]
where $C(X)[\text{tor}]\subset C(X)$ is the torsion subgroup. Note that our $d'_f$ differs from the integer $d'$ of \cite{raynaud}, but they are equal if $X$ is locally factorial at the points of $X_t$, cf. \cite[6.1.11 (1)]{raynaud}.

\begin{lemma}\label{torsioninvariant}
In addition to our standing hypotheses, assume $f$ is proper and $f_*\mathcal{O}_X=\mathcal{O}_T$.
\begin{enumerate}[label=\normalfont(\roman*)]
\item $d'_f=d_f$ is the (apparent) multiplicity of $X_t$ (\S\ref{multiplicity}).
\item Assume $X$ is locally factorial at the points of $X_t$. If $g:Y\to T$ is another proper flat $T$-scheme that is locally factorial at the points of $Y_t$ and such that $g_*\mathcal{O}_Y=\mathcal{O}_T$ and $Y_u\simeq X_u$, then $d_f=d_g$.
\end{enumerate}
\end{lemma}
\begin{proof}
Note that $H^0(X_u,\mathcal{O}_{X_u}^*)=k(u)^*$, hence $C(X)=(\oplus_{1\leq i\leq c}\mathbb{Z})/\langle\Div_X(\pi)\rangle$. So $C(X)[\text{tor}]$ is a cyclic group; let $D=\sum_{i=1}^ca_i[X_i]$ be a generator. Consider the exact sequence
\[ 0\to\mathbb{Z}\overset{\cdot D}{\to} H^0(X,\oplus_{1\leq i\leq c}\eta_{i,*}\mathbb{Z})\to B\to 0 \]
Taking the tensor product with $\mathbb{Z}/p$ for a prime $p$ we obtain an exact sequence
\[ 0\to B[p]\to\mathbb{Z}/p\to\oplus_{1\leq i\leq c}\mathbb{Z}/p \]
Since $D$ is a generator of $C(X)[\text{tor}]$, we see that $B[p]=0$ if $p\mid d'_f$. Thus, $B[\text{tor}]$ is annihilated by an integer prime to $d'_f$, hence $B[\text{tor}]$ is a cyclic group of order $m$ with $(m,d'_f)=1$. Now, a diagram chase shows that $B[\text{tor}]\cong \langle\Div_X(\pi)\rangle/\langle d'_fD\rangle$, so $m\Div_X(\pi)=n d'_fD=\sum_ind'_fa_i[X_i]$ for some integer $n$. By Lemma \ref{coprime} (iii), we see that $d'_f$ divides all of the multiplicities of $X_t$, so $d'_f\mid d_f$. Conversely, by Lemma \ref{coprime} (iii), we can write $\Div_X(\pi)=d_f(\sum_ib_i[X_i])$ for some $b_i\in\mathbb{N}$, so $C(X)$ has an element of order $d_f$. Thus, $d_f\mid d'_f$ and this proves (i).

For (ii), first recall (\cite[\S1]{brauer2}) that, by local factoriality, the map $\Div_X$ induces an isomorphism
\[ j_*\mathcal{O}_{X_u}^*/\mathcal{O}_X^*\cong \oplus_{1\leq i\leq c}\eta_{i,*}\mathbb{Z} \]
In particular, the right hand side is functorial in (locally factorial) $X$ (without any further assumptions on $X$). Now, since $X$ is locally factorial, hence normal at the points of $X_t$ and $g$ proper, there is an open subset $U\subset X$ with $\codim_X(X\setminus U)\geq 2$ and a unique morphism $U\to Y$ extending the isomorphism on the generic fibres. Note that $H^0(U,\mathcal{O}_U)=H^0(X,\mathcal{O}_X)=\mathcal{O}(T)$ and $C(U)=C(X)$.

Let $Y_1,...,Y_d$ be the irreducible components of $Y_t$ and $\xi_i\in Y_i$ the generic points.  By the above, the morphism $U\to Y$ yields a commutative diagram with exact rows
\[ \begin{tikzcd}[column sep=small]
0 \arrow[r] & H^0(Y_u,\mathcal{O}_{Y_u}^*)/H^0(Y,\mathcal{O}_Y^*) \arrow[r] \arrow[d,equal] & H^0(Y,\oplus_{1\leq i\leq d}\xi_{i,*}\mathbb{Z})\arrow[r] \arrow[d] & C(Y) \arrow[r] \arrow[d] & 0 \\
0 \arrow[r] & H^0(X_u,\mathcal{O}_{X_u}^*)/H^0(U,\mathcal{O}_U^*) \arrow[r] & H^0(U,\oplus_{1\leq i\leq c}\eta_{i,*}\mathbb{Z})\arrow[r] & C(U) \arrow[r] & 0
\end{tikzcd} \]
thus, $\ker(C(Y)\to C(U)=C(X))$ is torsion free. So, $C(Y)[\text{tor}]\subset C(X)[\text{tor}]$ and $d'_g\mid d'_f$. Inverting the roles of $X$ and $Y$ shows $d'_f\mid d'_g$, hence $d'_f=d'_g$ and (ii) follows from (i).
\end{proof}

Let $D$ be a divisorial cycle supported on $X_t$. Define a presheaf $\mathcal{O}(D)$ by
\[ \mathcal{O}(D)(U)=\{s\in j_*\mathcal{O}_{X_u}(U):\Div_U(s)+D|_U\geq 0\} \]
This definition also makes sense for divisorial cycles with rational coefficients (briefly, \emph{rational divisorial cycles}). Moreover, by Lemma \ref{integrallyclosed}, we have $\mathcal{O}(0)=\mathcal{O}_X$.

\begin{lemma}\label{rationalcycles}
Let $D$ and $E$ be rational divisorial cycles supported on $X_t$.
\begin{enumerate}[label=\normalfont(\roman*)]
\item $\mathcal{O}(D)$ is a coherent sheaf of $\mathcal{O}_X$-modules.
\item $\mathcal{O}(D)=\mathcal{O}(\lfloor D\rfloor)$, where $\lfloor \cdot\rfloor$ denotes the floor function, applied componentwise to $D$.
\item If $E\leq D$, then $\mathcal{O}(E)\subset\mathcal{O}(D)$. In particular, if $E\leq 0$, then $\mathcal{O}(E)\subset\mathcal{O}_X$.
\item If $s\in j_*\mathcal{O}_{X_u}(X)^*$ and $D=\Div_X(s)$, then $\mathcal{O}(-D)=s\mathcal{O}_X$. Conversely, if $\mathcal{O}(D)$ is a line bundle, then, locally on $X$, we have $\lfloor D\rfloor=\Div_X(s)$ for some $s\in j_*\mathcal{O}_{X_u}(X)^*$.
\item There is a canonical map
\[ \mathcal{O}(D)\otimes_{\mathcal{O}_X}\mathcal{O}(E)\to\mathcal{O}(D+E) \]
It is an isomorphism if $D$ is integral and $\mathcal{O}(D)$ is a line bundle.
\end{enumerate}
\end{lemma}
\begin{proof}
One proves that $\mathcal{O}(D)$ is a sheaf as in the proof of Lemma \ref{integrallyclosed}. Clearly, it is an $\mathcal{O}_X$-module. To check that it is coherent is local so we may assume $X=\Spec(A)$ is affine. Let $a\in A$, $B=A[z]/(az-1)$, $U=\Spec(B)$, $M=\mathcal{O}(D)(X)$.

We first show that $\mathcal{O}(D)$ is quasi-coherent. It suffices to show that the canonical map $M\otimes_AB\to\mathcal{O}(D)(U)$ is an isomorphism (cf. \cite[I, 1.4.1]{ega}). Note that the map is injective, since $j_*\mathcal{O}_{X_u}$ is quasi-coherent and contains $\mathcal{O}(D)$. For the surjectivity, we need to show that for any $b\in\mathcal{O}(D)(U)$ there is $N\geq 0$ such that $a^Nb$ is the image an element of $M$. By the quasi-coherence of $j_*\mathcal{O}_{X_u}$, $a^Nb$ is the image of $s\in j_*\mathcal{O}_{X_u}(X)$ for $N\gg 0$. Write $D=\sum n_i[X_i]$. If $v_i(a)=0$, then $\eta_i\in U$ and $v_i(b)\geq -n_i$. If $v_i(a)>0$, then for $N\gg 0$ we have $v_i(a^Nb)=Nv_i(a)+v_i(b)\geq -n_i$. Thus, $v_i(s)\geq -n_i$ for all $1\leq i\leq c$, so $\Div_X(s)\geq -D$ and $s\in M$, as required.

For coherence, let $\pi\in\mathcal{O}(T)$ be a uniformizer and choose $N\geq 0$ such that $\Div_X(\pi^N)\geq D$. Then, $\pi^N\mathcal{O}(D)\subset\mathcal{O}_X$ by Lemma \ref{integrallyclosed}, so, since $\mathcal{O}(D)$ is $\pi$-torsion free and $X$ is noetherian, (i) follows (cf. \cite[I, 1.5.1]{ega}). The remaining statements are left to the reader.
\end{proof}

\subsection{Proof of Theorem \ref{gabberraynaudindex}}
We begin with some general considerations. Assume $f$ is flat, of finite type, with nonempty special fibre, and satisfies (N). Let $\pi\in\mathcal{O}(T)$ be a uniformizer.

Consider the divisorial cycle $D=\Div_X(\pi)=\sum_in_i[X_i]$. Since $X_t$ is $S_1$, for $N\gg 0$ the sheaf $\mathcal{O}(-D/N)$ is the radical of $\mathcal{O}(-D)$, i.e., the ideal sheaf of $X_0:=(X_k)_{\red}$. We fix such an integer $N\geq 1$ that is divisible by all $n_i$.

\begin{proposition}\label{vanishingcriterion}
Assume $f$ proper and $f_*\mathcal{O}_X=\mathcal{O}_T$. The following are equivalent
\begin{enumerate}[label=\normalfont(\roman*)]
\item $H^0(X_t,\mathcal{O}_{X_t})$ is not reduced
\item $H^0(X,\mathcal{O}(-D/N)/\mathcal{O}(-D))\neq 0$
\item there is an effective Cartier divisor $C\subset X$ such that $rC=D$ for some rational number $r>1$ and $\mathcal{O}(-C)|_C\simeq\mathcal{O}_C$.
\end{enumerate}
Moreover, if this holds, then in \textup{(iii)} we can take $r=p^e/m$ with $e,m\in\mathbb{N}$ and $(m,p)=1$.
\end{proposition}

We now work towards the proof of Proposition \ref{vanishingcriterion}. The filtration
\begin{equation} \begin{aligned}
\mathcal{O}(-D)=\mathcal{O}(-ND/N)\subset\mathcal{O}(-(N-1)D/N)\subset\cdots\\
\cdots\subset\mathcal{O}(-2D/N)\subset\mathcal{O}(-D/N)\subset\mathcal{O}
\end{aligned}\label{powerfiltration}
\end{equation}
induces a finite descending filtration of $H:=H^0(X,\mathcal{O}(-D/N)/\mathcal{O}(-D))$ with graded
\begin{equation}\label{graded}
\gr^aH:=\dfrac{H^0(X,\mathcal{O}(-aD/N)/\mathcal{O}(-D))}{H^0(X,\mathcal{O}(-(a+1)D/N)/\mathcal{O}(-D))}\subset H^0(X,\mathcal{F}_a)
\end{equation}
where
\[ \mathcal{F}_a:=\frac{\mathcal{O}(-aD/N)}{\mathcal{O}(-(a+1)D/N)} \]
for $0<a<N$ and $\mathcal{F}_a=0$ otherwise. For $0<a<N$ define
\[ \mathcal{G}_a:=\frac{\mathcal{O}(-aD)}{\mathcal{O}(-(Na+1)D/N)} \]
and a map of sheaves of sets
\begin{align*}
\varphi_N:\mathcal{F}_a &\to \mathcal{G}_a \\
s &\mapsto s^N
\end{align*}
This map is well defined (exercise). We are going to analyze the sections of $\mathcal{F}_a$ using the map $\varphi_N$ and deduce Proposition \ref{vanishingcriterion}.

\begin{lemma}\label{nonzero}
$\varphi_N$ maps nonzero sections to nonzero sections.
\end{lemma}
\begin{proof}
The claim is local. Suppose $\varphi_N(s)=0$ for $s\in\mathcal{O}(-aD/N)$. Then $s^N\in\mathcal{O}(-(Na+1)D/N)$, so $\Div_X(s)\geq (Na+1)D/N^2$. Hence, for all $1\leq i\leq c$ we have $v_i(s)\geq (Na+1)n_i/N^2$. Since $N$ is divisible by $n_i$, we have
\[ \left\lceil\frac{(Na+1)n_i}{N^2}\right\rceil=\left\lceil\frac{\lceil(Na+1)/N\rceil}{(N/n_i)}\right\rceil=\left\lceil\frac{\lceil a+1/N\rceil}{(N/n_i)}\right\rceil=\left\lceil\frac{a+1}{(N/n_i)}\right\rceil \]
(for the first equality, cf. \cite[\S 3.2, (3.10)]{concretemath}). Hence, $\lceil (Na+1)n_i/N^2\rceil= \lceil (a+1)n_i/N\rceil$ and, since $v_i(s)\in\mathbb{Z}_{\infty}$, we get $v_i(s)\geq (a+1)n_i/N$. Thus, $\Div_X(s)\geq (a+1)D/N$ and $s\in\mathcal{O}(-(a+1)D/N)$, so $s$ has image 0 in $\mathcal{F}_a$.
\end{proof}

\begin{lemma}\label{trivial1}
$\mathcal{G}_a\cong\mathcal{O}(-aD)|_{X_0}$
\end{lemma}
\begin{proof}
The canonical map $\mathcal{O}(-D/N)\otimes_{\mathcal{O}_X}\mathcal{O}(-aD)\to\mathcal{O}(-(Na+1)D/N)$ (Lemma \ref{rationalcycles} (v)) gives an exact sequence of $\mathcal{O}_{X_0}$-modules
\[ 0\to J\to\mathcal{O}(-aD)|_{X_0}\to \mathcal{G}_a\to 0 \]
We have
\[ \left\lceil\frac{(Na+1)n_i}{N}\right\rceil=\left\lceil an_i+\frac{n_i}{N}\right\rceil=an_i+\left\lceil\frac{n_i}{N}\right\rceil = an_i+1 \]
so $(\mathcal{O}(-D/N)\otimes_{\mathcal{O}_X}\mathcal{O}(-aD))_{\eta_i}\cong\mathcal{O}(-(Na+1)D/N)_{\eta_i}$, hence $J_{\eta_i}=0$ for all $1\leq i\leq c$. Since $X_0$ is reduced, $\mathcal{O}(-aD)|_{X_0}\cong\mathcal{O}_{X_0}$ injects into the product of its localisations at the points $\eta_i$, hence $J=0$.
\end{proof}

\begin{corollary}\label{nozeros}
Assume $H^0(X_0,\mathcal{O}_{X_0})$ is a field. If $0\neq s\in H^0(X,\mathcal{F}_a)$, then
\begin{enumerate}[label=\normalfont(\roman*)]
\item $\varphi_N(s)\in\mathcal{G}_a(X)$ vanishes nowhere on $X_t$
\item $s\in\mathcal{F}_a(X)$ vanishes nowhere on $X_t$.
\end{enumerate}
\end{corollary}
\begin{proof}
By Lemma \ref{trivial1} we have $H^0(X_0,\mathcal{G}_a)\simeq H^0(X_0,\mathcal{O}_{X_0})$. Since the latter is a field by assumption, (i) follows from Lemma \ref{nonzero}.

For (ii), we claim that if $s\in I\mathcal{F}_a$ for some ideal $I\subset\mathcal{O}_X$, then $\varphi_N(s)\in I\mathcal{G}_a$. It is enough to show that if $s\in I\mathcal{O}(-aD/N)$, then $s^N\in I\mathcal{O}(-aD)$. This is clear if $s=xt$ with $x\in I$ and $t\in\mathcal{O}(-aD/N)$, so by induction it suffices to show it for $xt+w$ with $w\in I\mathcal{O}(-aD/N)$ such that $w^N\in I\mathcal{O}(-aD)$. By the binomial formula, this reduces to showing $t^iw^{N-i}\in\mathcal{O}(-aD)$, which is clear.

In particular, if $s\in I\mathcal{F}_a$, then $\varphi_N(s)$ vanishes on the support of $\mathcal{O}_X/I$. By (i), this implies that the support of $\mathcal{O}_X/I$ does not meet $X_t$, hence (ii).
\end{proof}

\begin{corollary}\label{cartier}
Assume $H^0(X_0,\mathcal{O}_{X_0})$ is a field. If $H^0(X,\mathcal{F}_a)\neq 0$, then $\mathcal{O}(-aD/N)$ is a line bundle satisfying $\mathcal{O}(-aD/N)^{\otimes N}\cong\mathcal{O}(-aD)$.
\end{corollary}
\begin{proof}
We first work locally on $X_t$. Let $0\neq s\in H^0(X,\mathcal{F}_a)$ and choose a local lift $\tilde{s}\in\mathcal{O}(-aD/N)$ of $s$. Then $\tilde{s}^N\in\mathcal{O}(-aD)=\pi^a\mathcal{O}_X$, so $\tilde{s}^N=\pi^aw$ for some $w\in\mathcal{O}_X$. Since $\tilde{s}^N$ vanishes nowhere on $X_t$ (Corollary \ref{nozeros} (i)), up to localizing we may assume $w\in\mathcal{O}_X^*$; in particular, $\tilde{s}$ is not a zerodivisor on $X$ (as this holds for $\pi$). So, for any $g\in\mathcal{O}(-aD/N)$, we have $g/\tilde{s}=(g\tilde{s}^{N-1}w^{-1})\pi^{-a}\in j_*\mathcal{O}_{X_u}$, hence $g/\tilde{s}\in\mathcal{O}_X$ by Lemma \ref{integrallyclosed}. Thus, $\mathcal{O}(-aD/N)\subset\mathcal{O}_X$ is the free $\mathcal{O}_X$-submodule generated by $\tilde{s}$, hence the canonical map $\mathcal{O}(-aD/N)^{\otimes N}\to\mathcal{O}(-aD)$ is an isomorphism (Lemma \ref{rationalcycles} (v)) at the points of $X_t$. Since it is clearly an isomorphism over $X_u$, this implies the assertion.
\end{proof}

\begin{corollary}\label{trivial2}
Assume $H^0(X_0,\mathcal{O}_{X_0})$ is a field. The following are equivalent
\begin{enumerate}[label=\normalfont(\roman*)]
\item $H^0(X,\mathcal{F}_a)\neq 0$
\item $\mathcal{F}_a\cong\mathcal{O}(-aD/N)|_{X_0}$ is a trivial line bundle on $X_0$.
\end{enumerate}
\end{corollary}
\begin{proof}
Assume (i). As in the proof of Lemma \ref{trivial1}, we have an exact sequence
\[ 0\to J\to \mathcal{O}(-aD/N)|_{X_0}\to\mathcal{F}_a\to 0 \]
On the other hand,
\begin{equation}\label{inequality} \left\lceil\frac{(a+1)n_i}{N}\right\rceil = \left\lceil \frac{an_i}{N}+\frac{n_i}{N}\right\rceil\leq \left\lceil \frac{an_i}{N}\right\rceil+\left\lceil\frac{n_i}{N}\right\rceil=\left\lceil \frac{an_i}{N}\right\rceil +1
\end{equation}
If equality does not hold in (\ref{inequality}), then $\lceil (a+1)n_i/N\rceil=\lceil an_i/N\rceil$, hence $(\mathcal{F}_a)_{\eta_i}=0$. However, $\mathcal{F}_a$ has a global section vanishing nowhere on $X_0$ by Corollary \ref{nozeros}, hence
\[ 0\neq \mathcal{F}_a\otimes_{\mathcal{O}_X}k(\eta_i)=(\mathcal{F}_a)_{\eta_i}\otimes_{\mathcal{O}_{X,\eta_i}}k(\eta_i) \]
a contradiction. Thus, equality holds in (\ref{inequality}) and it follows that $J_{\eta_i}=0$ for all $1\leq i\leq c$. From Corollary \ref{cartier} we deduce $J=0$, hence $\mathcal{F}_a\cong\mathcal{O}(-aD/N)|_{X_0}$ is a line bundle on $X_0$. Now, let $0\neq s\in H^0(X,\mathcal{F}_a)$ and consider the exact sequence
\[ \mathcal{O}_{X_0}\overset{s}{\to}\mathcal{F}_a\to C\to 0 \]
Since $s$ has no zeros (Corollary \ref{nozeros} (ii)) and $\mathcal{F}_a$ is a line bundle, the fibres of $C$ are zero, hence $C=0$ by Nakayama's lemma, proving (ii). The converse is trivial.
\end{proof}

Recall that we have written $D=\Div_X(\pi)=\sum_in_i[X_i]$; then, by Lemma \ref{coprime} (iii), $\gcd(n_1,...,n_c)=d_f$ is the (apparent) multiplicity of $X_t$.

\begin{lemma}\label{cartierdivisor}
Assume $f$ is proper and $f_*\mathcal{O}_X=\mathcal{O}_T$. If $H^0(X,\mathcal{F}_a)\neq 0$, then $aD/N$ is an integral divisorial cycle and there are positive integers $b,e$ with $b<p^e$ and $(b,p)=1$, and an integral divisorial cycle $E$ such that
\begin{enumerate}[label=\normalfont(\roman*)]
\item $p^eE=D$ and $bE=aD/N$
\item $\mathcal{O}_X(-bE)=\mathcal{O}(-aD/N)$ is a line bundle with trivial restriction to $X_0$.
\end{enumerate}
\end{lemma}
\begin{proof}
Note that, since $f$ is proper and $f_*\mathcal{O}_X=\mathcal{O}_T$, $f$ has connected fibres (\cite[III\textsubscript{1}, 4.3.2]{ega}), hence $H^0(X_0,\mathcal{O}_{X_0})$ is a field. By Corollary \ref{cartier}, $\mathcal{L}:=\mathcal{O}(-aD/N)$ is a line bundle of finite order dividing $N$, which has trivial restriction to $X_0$ by Corollary \ref{trivial2}.

First assume $T$ is strictly local. Let $n\mid N$ be the order of $\mathcal{L}$. Write $n=qn'$ with $(n',p)=1$ and $q=p^{e_1}$ for some $e_1\in\mathbb{N}$. Since $T$ is strictly local and $f$ proper, the Kummer sequence implies $H^1(X,\mu_{n'})\cong\Pic(X)[n']$ and similarly for $X_0$. From the proper base change theorem (\cite[Exp. XII, 5.5]{sga4.3}), it follows that $\Pic(X)[n']\cong\Pic(X_0)[n']$. By Corollary \ref{trivial2}, we deduce that $\mathcal{L}^{\otimes q}$ is a trivial line bundle, hence $n=q$.

Since $H^0(X,\mathcal{L}^{\otimes q})\subsetneq H^0(X,\mathcal{O}_X)=\mathcal{O}(T)$, it follows that $\mathcal{L}^{\otimes q}=\pi^m\mathcal{O}_X$ for some positive integer $m$. Writing $N=ql$ for some positive integer $l$, by Corollary \ref{cartier} we have $ml=a<N$, hence $m<q$. In particular, $q\nmid m$.

Now, returning to the case of arbitrary $T$, for each $1\leq i\leq c$, let $a_i=v_i(s_i)$, where $s_i\in\mathcal{L}_{\eta_i}$ is a generator. Since $s_i^q$ generates $\pi^m\mathcal{O}_{X,\eta_i}$ (this can be checked over the strict henselisation of $\mathcal{O}(T)$), we have $qa_i=mn_i$. Multiplying by $l$ and dividing by $N$, we find $a_i=an_i/N$. Thus, $aD/N$ is an integral divisorial cycle.

Let $e_2\in\mathbb{N}$ be such that $p^{e_2}\parallel m$. Then, $e_1>e_2$ and $p^{e_1-e_2}\mid n_i$ for $1\leq i\leq c$, hence $p^{e_1-e_2}\mid d_f$. So, setting $e:=e_1-e_2$ and $b:=\frac{m}{p^{e_2}}$, $E:=\frac{1}{p^e}D$ is an integral divisorial cycle with the required properties.
\end{proof}

\begin{corollary}\label{cartierdivisortrivial}
Assume $f$ proper and $f_*\mathcal{O}_X=\mathcal{O}_T$. If $p\nmid d_f$, then $H^0(X,\mathcal{F}_a)=0$ for $0<a<N$.
\end{corollary}
\begin{proof}
Indeed, if $H^0(X,\mathcal{F}_a)\neq 0$, then $p^e\mid d_f$ by Lemma \ref{cartierdivisor} for some integer $e>0$.
\end{proof}

\begin{proof}[Proof of Proposition \ref{vanishingcriterion}]
For a divisorial cycle $E\geq 0$, we will also denote by $E$ the closed subscheme of $X$ of ideal sheaf $\mathcal{O}(-E)$.

Let $H:=H^0(X,\mathcal{O}(-D/N)/\mathcal{O}(-D))$. Then,  by (\ref{graded}), $H$ has a finite filtration with $\gr^aH\subset H^0(X,\mathcal{F}_a)$. If (i) holds, then the exact sequence
\[ 0\to\mathcal{O}(-D/N)/\mathcal{O}(-D)\to\mathcal{O}_{D}\to\mathcal{O}_{D/N}\to 0 \]
implies (ii), hence $\gr^aH\neq 0$ for some $0<a<N$. Then, by Lemma \ref{cartierdivisor}, $aD/N$ is an integral divisiorial cycle and $\mathcal{O}(-aD/N)$ is a line bundle. By Lemma \ref{rationalcycles} (v), we have
\[ \mathcal{O}(-aD/N)/\mathcal{O}(-D)\cong\mathcal{O}(-aD/N)|_{D-aD/N} \]
and, by Corollary \ref{nozeros} (ii), a nonzero element of $\gr^aH$ provides a nowhere vanishing global section, hence $\mathcal{O}(-aD/N)|_{D-aD/N}$ is a trivial line bundle. Since
\[ \mathcal{O}(-(N-a)D/N)\cong\mathcal{O}(-aD/N)^{\vee}\otimes_{\mathcal{O}_X}\mathcal{O}(-D) \]
(Lemma \ref{rationalcycles} (v)), it follows that $C:=(N-a)D/N$ satisfies $\mathcal{O}(-C)|_C\simeq\mathcal{O}_C$. Moreover, by Lemma \ref{cartierdivisor} we have $rC=D$, where $r=p^e/(p^e-b)>1$, hence (iii). Note that, since $1\leq b<p^e$ and $(b,p)=1$, $m:=p^e-b$ is a positive integer not divisible by $p$.

Now assume (iii) and let us deduce (i). Let $q,n\in\mathbb{N}$ be such that $r=q/n$. Then, $qC=nD$, so $\mathcal{O}(-C)^{\otimes q-1}|_C\subset\mathcal{O}_{nD}$ and, if $s\in H^0(C,\mathcal{O}(-C)|_C)$ is a generator, then $0\neq s^{q-1}\in H^0(X,\mathcal{O}_{nD})$. We claim that $s^{q-1}=\pi^{n-1}g$ for some $g\in H^0(X,\mathcal{O}_D)$. Let $\tilde{s}\in\mathcal{O}(-C)$ be a local lift; then $\tilde{s}^q=\pi^nw$ for some unit $w\in\mathcal{O}_X^*$. So, $\tilde{g}:=w\pi/\tilde{s}=\tilde{s}^{q-1}\pi^{1-n}\in j_*\mathcal{O}_{X_u}$ satisfies $\Div_X(\tilde{g})=D-C=(r-1)C>0$, hence $\tilde{g}\in\mathcal{O}_X$. Taking $g$ to be the reduction of $\tilde{g}$ to $\mathcal{O}_D$, we have $s^{q-1}=\pi^{n-1}g$ as claimed. Thus, $s^{q-1}$ is a global section of the subsheaf $\pi^{n-1}\mathcal{O}_X/\pi^n\mathcal{O}_X\subset\mathcal{O}_{nD}$. Since multiplication by $\pi^{n-1}$ induces an isomorphism $\mathcal{O}_D\cong\pi^{n-1}\mathcal{O}_X/\pi^n\mathcal{O}_X$, this implies $g\in H^0(X,\mathcal{O}_D)$. Moreover, we have $\tilde{g}^q=w^{q-1}\pi^{n-q}$, hence $g^q=\pi^{q-n}\bar{w}^{q-1}$, where $\bar{w}\in\mathcal{O}_D^*$ is the reduction of $w$. Since $r>1$, we have $q-n>0$, so $g^q=0$ in $\mathcal{O}_D$, whence (i).
\end{proof}

\begin{proof}[Proof of Theorem \ref{gabberraynaudindex}]
We first reduce to the case $f_*\mathcal{O}_X=\mathcal{O}_T$. By Lemma \ref{steinreduction}, $H^0(X,\mathcal{O}_X)$ is a finite normal $\mathcal{O}(T)$-algebra. Note that $H^0(X,\mathcal{O}_X)\otimes_{\mathcal{O}(T)}k(t)$ is connected since $X_t$ is, hence $H^0(X,\mathcal{O}_X)$ is local, whence a discrete valuation ring. This implies that the generic fibre of $X\to\Spec(H^0(X,\mathcal{O}_X))$ is $X_u$ and its special fibre is homeomorphic to $X_t$. It follows that the index of the generic fibre of $X\to\Spec(H^0(X,\mathcal{O}_X))$ divides that of the generic fibre of $X\to T$, hence is prime to $p$. Thus, by Lemma \ref{steinreduction}, we may assume $f_*\mathcal{O}_X=\mathcal{O}_T$; then $f$ has geometrically connected fibres.

Now, by the assumption on the index, there is some closed point $x\in X_u$ such that $[k(x):k(u)]$ is not divisible by $p$. In particular, $k(u)\subset k(x)$ is separable, so the normalization $T'$ of $T$ in $\Spec(k(x))$ is finite over $T$. The degree of $T'/T$ being prime to $p$ implies that there must be some closed point $t'\in T'$ whose ramification index and residual degree are prime to $p$. So, over the strict localization $T^{\sh}$ of $T$, we deduce that there exists a point of $X$ with values in a finite, tamely ramified extension of $T^{\sh}$.

Now, to show that $f$ is cohomologically flat we may, by Proposition \ref{basechange} (ii) and Lemma \ref{coprime} (vi), assume $T=T^{\sh}$. Let $T'\to T$ be a finite tamely ramified extension such that $X(T')\neq\emptyset$ and $t'\in T'$ the closed point. A quasi-section $\sigma:T'\to X$ induces a morphism $t'\to X_0$, hence a homomorphism $H^0(X_0,\mathcal{O}_{X_0})\to k(t')=k(t)$. Since $X_t$ is connected, $H^0(X_0,\mathcal{O}_{X_0})$ is a field, so $H^0(X_0,\mathcal{O}_{X_0})\subset k(t)$. Hence, $H^0(X_0,\mathcal{O}_{X_0})=k(t)$ and it suffices to show that $H^0(X_t,\mathcal{O}_{X_t})$ is reduced, for then it must equal $k(t)$ and $H^0(X,\mathcal{O}_X)\to H^0(X_t,\mathcal{O}_{X_t})$ is surjective, so $f$ is cohomologically flat by Proposition \ref{basechange}. If not, then, by Proposition \ref{vanishingcriterion}, $mD=p^eC$ for $C$ an effective Cartier divisor, where $(m,p)=1$ and $e>0$. Let $z=\sigma(t')\in X$. If $s\in\mathcal{O}(-C)_z$ is a local generator, then $s^{p^e}=\pi^nw$ for some unit $w\in\mathcal{O}_{X,z}^*$. Hence, the image of $s$ in $\mathcal{O}_{T',t'}$ has valuation $im/p^e$, where $i$ is the ramification index of $T'/T$. Since $(im,p)=1$, this forces $e=0$, a contradiction.
\end{proof}

We deduce a higher-dimensional generalization of Raynaud's Theorem \ref{raynaudintro}.

\begin{corollary}\label{raynaud2}
Let $f:X\to T$ be a proper and flat morphism satisfying (N). If the special fibre of $f$ is connected and $p\nmid\delta_f$, then $f$ is cohomologically flat.
\end{corollary}
\begin{proof}
Arguing as in the proof of Theorem \ref{gabberraynaudindex}, by Lemma \ref{coprime} (iv) we reduce to the case $f_*\mathcal{O}_X=\mathcal{O}_T$. By Proposition \ref{basechange} (ii), Lemma \ref{coprime} (vi) and Lemma \ref{complete}, we may also assume $T$ is strictly local. In this case, the index of $X_u/u$ divides $\delta_f$ by \cite[7.1.6 (1)]{raynaud}, hence is prime to $p$ and the result follows from Theorem \ref{gabberraynaudindex}.
\end{proof}

\begin{remark}
For $f$ projective and $S_2$, one can also prove Corollary \ref{raynaud2} by induction on the dimension of $X$, the base case being Raynaud's Theorem \ref{raynaudintro}, with a suitably chosen hyperplane section for the induction step.
\end{remark}

\begin{remark}
In the proof of Corollary \ref{raynaud2}, we used the fact the index of $X_u/u$ divides $\delta_f$ if $T$ is strictly local and $f_*\mathcal{O}_X=\mathcal{O}_T$. If, in addition, both $X$ and the irreducible components of $X_t$ are regular, then it follows from \cite[8.2]{GLL} that the index of $X_u/u$ is equal to $\delta_f$ (use \cite[7.1.1]{raynaud} to show that the index of an irreducible component of $X_t$ is equal to its multiplicity).
\end{remark}

We also show a variant with the index of the generic fibre replaced by the gcd of the multiplicities of the special fibre. Recall that $X_0=(X_t)_{\red}$.

\begin{theorem}\label{gabberraynaud}
Let $T$ be the spectrum of a discrete valuation ring and $f:X\to T$ a proper and flat morphism satisfying $(N)$. If $H^0(X_0,\mathcal{O}_{X_0})=k(t)$ and $p\nmid d_f$, then $f$ is cohomologically flat in dimension 0.
\end{theorem}
\begin{proof}
Since $H^0(X_0,\mathcal{O}_{X_0})=k(t)$, it suffices to show that $H^0(X_t,\mathcal{O}_{X_t})$ is reduced. Note that $X_t$ is connected since $H^0(X_0,\mathcal{O}_{X_0})$ is. Arguing as in the proof of Corollary \ref{raynaud2}, we reduce to the case $f_*\mathcal{O}_X=\mathcal{O}_T$. In this case, by Corollary \ref{cartierdivisortrivial} we have $H^0(X,\mathcal{F}_a)=0$ for $0<a<N$. Thus, $H^0(X,\mathcal{O}(-D/N)/\mathcal{O}(-D))=0$ by (\ref{graded}) and $H^0(X_t,\mathcal{O}_{X_t})$ is reduced by Proposition \ref{vanishingcriterion}.
\end{proof}

\begin{remark}
The first named author can generalize Theorem \ref{gabberraynaud} to formal schemes over (not necessarily discrete) valuation rings.
\end{remark}

\section{Logarithmic criterion}\label{logsection}
We begin with some preliminaries that will be of use later on in the section.

\begin{definition}\label{ncdscheme}
Let $S$ be a scheme. A \emph{normal crossings scheme over $S$} is a morphism of locally finite presentation $D\to S$ such that, locally for the \'etale topology, $D$ is the scheme-theoretic union of closed subschemes $D_1,...,D_r\subset D$ and there is an integer $d$ such that, for all $J\subset\{1,...,r\}$, $D_J:=\cap_{j\in J}D_j$ is a smooth $S$-scheme of relative dimension $d+1-|J|$.
\end{definition}

\begin{lemma}\label{semistablelemma}
Let $S\to T$ be a morphism of spectra of discrete valuation rings such that the closed fibre $S_t$ is the spectrum of a field. If $X$ is a regular, proper and flat $T$-scheme and $X_0:=(X_t)_{\red}$ is a normal crossings scheme over $t$, then $X\times_TS$ is regular and $X_0\times_TS\subset X\times_TS$ is a normal crossings divisor.
\end{lemma}
\begin{proof}
Note that $S\to T$ is flat. We claim that $X\times_TS$ is regular at the points of its special fibre. This is local for the \'etale topology on $X$ and so we may assume $X_0=\cup_{i=1}^rX_i$ with $X_i\to t$ smooth. Then $X_i\times_TS$ is smooth over $S_t$, hence regular and, since $X$ is regular and $S\to T$ is flat, the closed immersions $X_i\times_TS\to X\times_TS$ are regular (\cite[IV\textsubscript{4}, 19.1.5]{ega}). So $X\times_TS$ is regular along $X_i\times_TS$ for $1\leq i\leq r$ (\cite[IV\textsubscript{4}, 19.1.1]{ega}), hence the claim. Now, since the regular locus of an $S$-scheme of finite type is open (\cite[IV\textsubscript{2}, 6.12.6]{ega}) and $X\times_TS\to S$ is proper, it follows that $X\times_TS$ is regular. We leave the other claim to the reader.
\end{proof}

\begin{lemma}\label{acreduction}
Let $T$ be a noetherian local scheme of closed point $t$ and $M_T$ an fs Zariski log structure on $T$ such that $(T,M_T)$ is log regular. For a field extension $k(t)\subset k(s)$, there is strict morphism of log regular schemes $(S,M_{S})\to (T,M_T)$ with $S$ local noetherian and $S\to T$ faithfully flat such that $k(s)\cong\mathcal{O}(S)\otimes_{\mathcal{O}(T)}k(t)$. Then,
\begin{enumerate}[label=\normalfont(\roman*)]
\item $f_{\log}$ is log smooth if and only if $f_{\log}\times_{(T,M_T)}(S,M_{S})$ is
\item if $f$ is proper and flat, $f$ is cohomologically flat if and only if $f\times_TS$ is.
\end{enumerate}
Moreover, $\mathcal{O}(S)$ may be taken to be henselian or complete.
\end{lemma}
\begin{proof}
Let $\mathfrak{m}\subset\mathcal{O}_{T,t}$ be the maximal ideal and $\mathfrak{p}\subset\mathfrak{m}$ the ideal generated by the image of $M_{T,t}\setminus M_{T,t}^*$. By log regularity, $\mathcal{O}_{T,t}/\mathfrak{p}$ is regular of dimension $\dim T-\rk\overline{M}_{T,t}^{\gp}$; in particular, $\mathfrak{m}/\mathfrak{p}$ can be generated by $\dim\mathcal{O}_{T,t}/\mathfrak{p}$ elements.

By \cite[0\textsubscript{III}, 10.3.1]{ega}, there exists a flat morphism of noetherian local schemes $S\to T$ such that $\mathcal{O}(S)/\mathfrak{m}\mathcal{O}(S)\cong k(s)$. Then, $\dim S=\dim T$. Let $M_S$ the inverse image log structure of $M_T$, so that $\overline{M}_{S,s}\cong\overline{M}_{T,t}$, where $s\in S$ is the closed point. We have a short exact sequence
\[ 0\to (\mathfrak{m}/\mathfrak{p})\otimes_{\mathcal{O}_{T,t}}\mathcal{O}_{S,s}\to\mathcal{O}_{S,s}/\mathfrak{p}\mathcal{O}_{S,s}\to k(s)\to 0 \]
hence
\[ \dim\mathcal{O}_{S,s}/\mathfrak{p}\mathcal{O}_{S,s}\leq\dim\mathcal{O}_{T,t}/\mathfrak{p} \]
Since the reverse inequality holds by flatness, the maximal ideal of $\mathcal{O}_{S,s}/\mathfrak{p}\mathcal{O}_{S,s}$ can be generated by $\dim\mathcal{O}_{T,t}/\mathfrak{p}=\dim\mathcal{O}_{S,s}/\mathfrak{p}\mathcal{O}_{S,s}$ elements, hence $\mathcal{O}_{S,s}/\mathfrak{p}\mathcal{O}_{S,s}$ is regular. Moreover,
\[ \dim\mathcal{O}_{S,s}/\mathfrak{p}\mathcal{O}_{S,s}=\dim\mathcal{O}_{T,t}/\mathfrak{p}=\dim S-\rk\overline{M}_{S,s}^{\gp} \]
so $(S,M_S)$ is log regular. Finally, since $S\to T$ is faithfully flat, (i) follows from \cite{logflatdescent} and (ii) from Proposition \ref{basechange} (ii).
\end{proof}

\subsection{Proof of Theorem \ref{main}}\label{proofmain}
To prove Theorem \ref{main} we may, by Proposition \ref{basechange}, assume $T$ to be strictly local of closed point $t\in T$ and then, by Lemma \ref{acreduction}, assume $k(t)$ to be algebraically closed.

Now, since $T$ is strictly local and $M_T$ fs, there exists a chart $P\to \Gamma(T,M_T)$ with $P\cong\overline{M}_{T,t}$; then, the image of $P\setminus\{1\}$ lies in the maximal ideal of $\mathcal{O}_{T,t}$. For a positive integer $n$, the homomorphism $P\to P^{(n)}:=P$ given by elevation to the $n$th power induces a morphism of schemes $\mathbf{n}:\Spec(\mathbb{Z}[P^{(n)}])\to \Spec(\mathbb{Z}[P])$. We denote by $T^{(n)}$ the base change of $T\to\Spec(\mathbb{Z}[P])$ by $\mathbf{n}$, i.e.,
\[ T^{(n)}:=T\times_{\Spec(\mathbb{Z}[P]),\mathbf{n}}\Spec(\mathbb{Z}[P^{(n)}]) \]
The standard log structure of the right-hand factor induces a log structure $M_{T^{(n)}}$ on $T^{(n)}$. There is a natural action of the finite diagonalizable group $G:=\D(P^{\gp}\otimes\mathbb{Z}/n)$ on the fs log scheme $(T^{(n)},M_{T^{(n)}})$, where $\D(-)$ denotes the Cartier dual---see \cite[\S3]{illusie}.

\begin{lemma}\label{logregularextension}
$(T^{(n)},M_{T^{(n)}})$ is log regular and $T^{(n)}$ is strictly local.
\end{lemma}
\begin{proof}
Left to the reader.
\end{proof}

Consider the fs log scheme
\[ \begin{split}
(X^{(n)},M_{X^{(n)}}) &:=(X,M_X)\times_{(T,M_T)}(T^{(n)},M_{T^{(n)}}) \\
&\cong (X,M_X)\times_{\Spec(\mathbb{Z}[P]),\mathbf{n}}\Spec(\mathbb{Z}[P^{(n)}]) 
\end{split} \]
By base change, the action of $G$ extends to $(X^{(n)},M_{X^{(n)}})$. Moreover, $f^{(n)}_{\log}:=f_{\log}\times_{(T,M_{T})}(T^{(n)},M_{T^{(n)}})$ is log smooth, hence, by Lemma \ref{logregularextension}, $X^{(n)}$ is normal and Cohen--Macaulay (\cite[II.4.7]{tsuji}). Denote by $f^{(n)}:X^{(n)}\to T^{(n)}$ the morphism underlying $f^{(n)}_{\log}$ and $q:X^{(n)}\to X$ the morphism underlying the projection.

\begin{lemma}\label{integral}
For $n\gg 0$, $f^{(n)}_{\log}$ is a saturated morphism.
\end{lemma}
\begin{proof}
Note that, since $(T,M_T)$ is log regular, $f_{\log}$ is log flat and $f$ is flat, $f_{\log}$ is integral by \cite[IV, 4.3.5 (3)]{logbook}. In other words, for any geometric point $x\to X$ with image $f(x)\to T$, the homomorphism $\overline{M}_{T,f(x)}\to\overline{M}_{X,x}$ is integral; in particular it is injective (cf. \cite[I, 4.6.7]{logbook}).

By definition, the assertion means that the homomorphism $\overline{M}_{T^{(n)},f^{(n)}(x')}\to\overline{M}_{X^{(n)},x'}$ is saturated for any geometric point $x'\to X^{(n)}$ (\cite[II.2.10]{tsuji}). Since this property is open in the \'etale topology (\cite[III, 2.5.2]{logbook}) and $X^{(n)}\to T$ is proper, we may assume $x'$ lies over $t$. Furthermore, since the property is fppf local on $X^{(n)}$, hence also on $X$, so we may assume there is a neat chart $P\to Q$ of $f_{\log}$ in an fppf neighbourhood of the image $x$ of $x'$ in $X$ (\cite[III, 1.2.7]{logbook}); since $\overline{M}_{T,t}\to\overline{M}_{X,x}$ is injective we have $Q\cong\overline{M}_{X,x}$ (\emph{loc.cit.}). Then, the pushout $Q\oplus_PP^{(n)}$ in the category of fs monoids provides a chart for $M_{X^{(n)}}$ at $x$ and we have $\overline{M}_{X^{(n)},x}\cong (Q\oplus_PP^{(n)})/A$, where $A\subset Q\oplus_PP^{(n)}$ is the torsion subgroup (\cite[2.1.1]{nakayamalogetale}).

Now, since $P\to Q$ is an integral homomorphism, $P^{(n)}\to Q\oplus_PP^{(n)}$ is the composition of an integral homomorphism with saturation, hence is $\mathbb{Q}$-integral (\cite[I, 4.7.4]{logbook}). By \cite[A.3.4]{illusie-kato-nakayama}, there is an integer $m\geq 1$ such that
\[ P^{(nm)}\cong P^{(n)}\oplus_{P^{(n)}}P^{(nm)}\to Q\oplus_PP^{(n)}\oplus_{P^{(n)}}P^{(nm)}\cong Q\oplus_{P}P^{(nm)} \]
is an integral homomorphism. Then, by \cite[I.5.4]{tsuji}, there is an integer $l\geq 1$ such that 
\[ P^{(nml)}\cong P^{(nm)}\oplus_{P^{(nm)}}P^{(nml)}\to Q\oplus_{P}P^{(nm)}\oplus_{P^{(nm)}}P^{(nml)}\cong Q\oplus_PP^{(nml)} \]
is a saturated homomorphism. Thus, without loss of generality we may assume $P^{(n)}\to Q\oplus_PP^{(n)}$ to be saturated; moreover, this holds for any multiple of $n$ (since a saturated homomorphism is stable under pushouts \cite[I.3.14]{tsuji}). Then, $P^{(n)}\to (Q\oplus_PP^{(n)})/A\cong\overline{M}_{X^{(n)},x}$ is also saturated by \cite[I, 4.8.5 (4)]{logbook} (since torsion elements of monoids are units), hence $f^{(n)}_{\log}$ is saturated at $x$. This proves the assertion locally for the fppf topology on $X$. Since $X$ is quasi-compact, a finite product of the integers $n$ for a covering works globally.
\end{proof}

\begin{lemma}\label{quotient}
\begin{enumerate}[label=\normalfont(\roman*)]
\item $X\cong X^{(n)}/G$, where the latter denotes the quotient ringed space.
\item $q_*\mathcal{O}_{X^{(n)}}$ is a $\D(G)$-graded $\mathcal{O}_X$-module and the canonical map $\mathcal{O}_X\to q_*\mathcal{O}_{X^{(n)}}$ is an isomorphism onto the submodule of degree zero.
\end{enumerate}
\end{lemma}
\begin{proof}
First note that $q:X^{(n)}\to X$ is finite and $(X^{(n)},M_{X^{(n)}})\to (X,M_X)$ is a Kummer log flat covering (since $(T^{(n)},M_{T^{(n)}})\to (T,M_T)$ is), hence $q$ is open and surjective (cf. \cite[\S 2]{logII}). In particular, the inverse image of an affine open covering of $X$ forms a covering of $X^{(n)}$ by $G$-stable affine open subsets. By \cite[Exp. V, 4.1]{sga3.1}, it follows that the quotient ringed space $X^{(n)}/G$ is a scheme and the quotient morphism $X^{(n)}\to X^{(n)}/G$ is open and integral, hence finite (since $X^{(n)}$ is of finite type over $X$, hence also over $X^{(n)}/G$). So the Artin--Tate lemma implies that $X^{(n)}/G\to X$ is of finite type, hence finite. Moreover, $X^{(n)}\to T$ is open since both $X^{(n)}\to X$ and $X\to T$ are; in particular, $X^{(n)}/G\to T$ is also open.

Now, over the dense open subset $U\subset T$ where $M_T$ is trivial, $T^{(n)}\times_TU\to U$ is a $G$-torsor, hence $(X^{(n)}/G)\times_TU\cong X\times_TU$. Also, since $X^{(n)}/G\to T$ is open, $(X^{(n)}/G)\times_TU\subset X^{(n)}/G$ is dense. Thus, the canonical map $X^{(n)}/G\to X$ is finite and birational. As $X^{(n)}/G$ is reduced (since $X^{(n)}$ is), (i) now follows from \cite[\S2.3, Thm. 2']{neronmodels}.

For (ii), let $D=\D(G)$. Then, $G=\Spec(\mathbb{Z}[D])$ and the $G$-action $G\times X^{(n)}\to X^{(n)}$ induces a canonical homomorphism
\[ q_*\mathcal{O}_{X^{(n)}}\to q_*\mathcal{O}_{X^{(n)}}\otimes\mathbb{Z}[D]\cong\oplus_{d\in D}q_*\mathcal{O}_{X^{(n)}} \]
which defines a $D$-grading on $q_*\mathcal{O}_{X^{(n)}}$, cf. \cite[Exp. I, 4.7.3]{sga3.1}. By (i), $\mathcal{O}_X\subset q_*\mathcal{O}_{X^{(n)}}$ is the subgroup of degree 0, hence (ii).
\end{proof}

\begin{proof}[Proof of Theorem \ref{main}]
Fix $n$ such that $f^{(n)}_{\log}$ is saturated (Lemma \ref{integral}). Then, $f^{(n)}$ is flat (\cite[4.5]{log}) and has reduced fibres by \cite[II.4.2]{tsuji}. Since $k(t)$ is algebraically closed, it follows from \cite[IV\textsubscript{3}, 12.2.1]{ega} that $f^{(n)}:X^{(n)}\to T^{(n)}$ has geometrically reduced fibres. So, by Proposition \ref{basechange} (iv), $f^{(n)}$ is cohomologically flat.

Let $g:T^{(n)}\to T$. Since $f^{(n)}$ is cohomologically flat we have
\[ (f^{(n)}_*\mathcal{O}_{X^{(n)}})\otimes_{\mathcal{O}_{T^{(n)}}}g^*k(t)\cong f^{(n)}_*\mathcal{O}_{X^{(n)}\times_Tt} \]
and, as $g:T^{(n)}\to T$ is affine, applying $g_*$ we obtain
\[ (g_*f_*^{(n)}\mathcal{O}_{X^{(n)}})\otimes_{\mathcal{O}_{T}}k(t)\cong g_*f_*^{(n)}\mathcal{O}_{X^{(n)}\times_{T}t} \]
by \cite[I, 9.3.2]{ega1}. Since $g\circ f^{(n)}=f\circ q$, we deduce that the natural map
\begin{equation}\label{first}
(f_*q_*\mathcal{O}_{X^{(n)}})\otimes_{\mathcal{O}_T}k(t)\to f_*q_*\mathcal{O}_{X^{(n)}\times_Tt}
\end{equation}
is an isomorphism. Since it is $G$-equivariant, it is an isomorphism of $\D(G)$-graded modules (cf. \cite[Exp. I, 4.7.3]{sga3.1}). By Lemma \ref{quotient} (ii), $q_*\mathcal{O}_{X^{(n)}}$ has degree zero submodule $\mathcal{O}_{X}$, hence $q_*\mathcal{O}_{X^{(n)}\times_Tt}\cong (q_*\mathcal{O}_{X^{(n)}})\otimes_{\mathcal{O}_X}\mathcal{O}_{X_t}$ has degree zero submodule $\mathcal{O}_{X_t}$. Taking degree zero submodules in (\ref{first}) we obtain
\[ (f_*\mathcal{O}_X)\otimes_{\mathcal{O}_T}k(t)\cong f_*
\mathcal{O}_{X_t} \]
By Proposition \ref{basechange} (v), this implies that $f$ is cohomologically flat.
\end{proof}

\subsection{An example}\label{examplesection}
Let $T$ be the spectrum of a discrete valuation ring and $M_T$ its canonical log structure. To illustrate the proof of Theorem \ref{main}, we give an example of a log smooth morphism to $(T,M_T)$ with nonreduced fibres, which acquires smooth fibres after base change $(T^{(n)},M_{T^{(n)}})\to (T,M_T)$, where $(T^{(n)},M_{T^{(n)}})$ is as in \S\ref{proofmain} with $P=\mathbb{N}$. Assume $n>1$ and let $A\to T$ be an abelian scheme containing $(\mu_n)_T$ as a subgroup. For example, this holds if $\mathcal{O}(T)$ is complete, $k(t)$ is algebraically closed, and $A_t$ is ordinary (cf. \cite[p. 150]{katzST}). Then, letting $\mu_n$ act on $A$ by translations, one can show that the contracted product $A\times_T^{\mu_n}T^{(n)}$ is representable by a regular projective $T$-scheme $X$ that is naturally log smooth over $(T,M_T)$. Moreover, the multiplicity of $X_t$ is $n>1$.

\subsection{Application: torsors under abelian varieties with good reduction}\label{torsorsection}
Let $T$ be the spectrum of a discrete valuation ring---see \S\ref{conventions} for notation. Let $A/T$ an abelian scheme and $Y$ an $A_u$-torsor.

\begin{definition}
An \emph{$A$-model} of $Y$ is a faithfully flat finite type $T$-scheme $X$ with an $A$-action $\mu:A\times_TX\to X$, such that
\begin{enumerate}[label=\normalfont(\alph*)]
\item the restriction of $(X,\mu)$ to $u$ is isomorphic to $Y$ with its $A_u$-action
\item the morphism $\mu\times\pr_X:A\times_TX\to X\times_TX$ is surjective.
\end{enumerate}
\end{definition}

By a theorem of Raynaud \cite[(c) p. 82]{raynaudquotient}, there exists a regular, projective $A$-model $X$ of $Y$, which, up to isomorphism, is the unique proper regular $A$-model (cf. \cite{lewin-menegaux}).

\begin{lemma}\label{model}
If $X$ is the proper regular $A$-model of an $A_u$-torsor, then $X_0:=(X_t)_{\red}$ is smooth over the field $k_0:=H^0(X_0,\mathcal{O}_{X_0})$. In fact, for $k_0\subset\bar{k}$ a separable closure, $X_0\otimes_{k_0}\bar{k}$ is isomorphic to the quotient of $A\otimes_{k(t)}\bar{k}$ by a finite subgroup scheme.
\end{lemma}
\begin{proof}
See \cite[8.1]{LLR} for the first statement. The second follows from \emph{loc.cit.} and \cite[Exp. V, 10.1.2]{sga3.1}.
\end{proof}

In particular, $X_0=(X_t)_{\red}$ is regular, so if we let $M_X$ be the canonical fs log structure $M_X$ defined by $X_0\subset X$, then $(X,M_X)$ is log regular (cf. Prop. \ref{ncd} below). Denote $f_{\log}:(X,M_X)\to (T,M_T)$ the resulting morphism.

\begin{theorem}\label{torsor}
The following are equivalent:
\begin{enumerate}[label=\normalfont(\roman*)]
\item $f_{\log}:(X,M_X)\to (T,M_T)$ is log smooth
\item $f:X\to T$ is cohomologically flat and $M_T$ is the canonical log structure.
\end{enumerate}
Moreover, in this case $X_0/t$ is smooth and its conormal sheaf in $X$ is a line bundle of order equal to the multiplicity of $X_t$.
\end{theorem}
\begin{proof}
By Theorem \ref{main}, we may assume $f$ is cohomologically flat. Note that $\mathcal{O}(T)=H^0(X,\mathcal{O}_X)$ (since $X$ is normal).

Let $I=\mathcal{O}_X(-X_0)$ be ideal sheaf of $X_0$ in $X$; it is a line bundle of order dividing the multiplicity $d:=d_f$ of $X_t$. For $n\in\mathbb{N}$, let $X_n\subset X$ be the closed subscheme of ideal sheaf $I^{n+1}$. Then $X_t=X_{d-1}$ and we have exact sequences
\[ 0\to I^n/I^{n+1}\to\mathcal{O}_{X_{n}}\to\mathcal{O}_{X_{n-1}}\to 0 \]
where $I^n/I^{n+1}\cong (I/I^2)^{\otimes n}$ is a line bundle on $X_0$.

Since $f$ is cohomologically flat, we have $k(t)=\Gamma(\mathcal{O}_{X_{d-1}})\hookrightarrow\Gamma(\mathcal{O}_{X_0})$, hence $\Gamma(I^{d-1}/I^d)=0$. It follows from \cite[II, \S8, (vii) p. 76]{mumfordAV} that $H^1(X_0,I^{d-1}/I^d)=0$ (cf. Lemma \ref{model}), so we obtain an isomorphism $k(t)=\Gamma(\mathcal{O}_{X_{d-1}})\cong\Gamma(\mathcal{O}_{X_{d-2}})$. Continuing this way, we get $\Gamma(I^n/I^{n+1})=0$ for $0<n<d$ and $\Gamma(\mathcal{O}_{X_0})=k(t)$. In particular, $(I/I^2)^{\otimes n}$ is not a trivial line bundle for $0<n<d$, so $I/I^2$ has order at least $d$, hence equal to $d$. Moreover, since $\Gamma(\mathcal{O}_{X_0})=k(t)$, $X_0$ is smooth and geometrically connected over $t$ by Lemma \ref{model}. In particular, $X_0$ is a normal crossings scheme (Def. \ref{ncdscheme}) over $t$. It follows from Lemma \ref{semistablelemma} and Lemma \ref{acreduction} that we may assume $k(t)$ algebraically closed and $T$ complete.

Now, if (i) holds and $M_T$ is the trivial log structure, then by log smoothness we have $X(T)\neq\emptyset$, hence $X\simeq A$ by uniqueness of the Raynaud model. In particular, $f$ is smooth and we leave it to the reader to check that $M_T$ cannot be trivial by definition of $M_X$. This proves (i) $\Rightarrow$ (ii).

For the converse, if $p\nmid d$, then $f_{\log}$ is log smooth by \cite[Prop. 1 (d)]{logsmoothcurves} (cf. Prop. \ref{ncd} (v) below). So assume $p\mid d$. Let $\nu^1=\mathbb{G}_{m}/p$. Given a uniformizer $\pi\in\mathcal{O}(T)$, locally on $X$ we may write $\pi=vx^d$ with $v\in\mathcal{O}_X^*$ and $x\in I$ a local generator. One checks without difficulty that $v$ extends to a global section of $H^0(X,\nu^1)$, hence defines a global section $\alpha\in H^0(X_0,\nu^1)$.

The exact sequence of (Zariski) sheaves on $X_0$
\[ 0\to\mathbb{G}_{m,X_0}\overset{p}{\to}\mathbb{G}_{m,X_0}\to\nu^1_{X_0}\to 0 \]
provides an isomorphism $H^0(X_0,\nu^1)\cong\Pic(X_0)[p]$. It is straightforward to check that this maps $\alpha$ to the class of $\frac{d}{p}(I/I^2)$, thus $\alpha\neq 0$.

On the other hand, since $X_0$ is smooth, the map $\dlog:\nu^1_{X_0}\to\Omega^1_{X_0/t}$ is injective, hence we obtain a nonzero global 1-form $\dlog(\alpha)$ on $X_0$. Now, a global 1-form on an abelian variety vanishes nowhere (cf. \cite[(iii) p. 42]{mumfordAV}), so the same is true for $\dlog(\alpha)$ (cf. Lemma \ref{model}). Finally, by \cite[Prop. 1 (c)]{logsmoothcurves}, $\Omega^1_{X_0/t}$ is a locally direct summand in $\omega^1_{X}|_{X_0}$ (where $\omega^1$ denotes the sheaf of logarithmic 1-forms), hence the image of $\dlog(\alpha)$ in the latter is a nowhere vanishing section. But this image is none other than the image of $\dlog(\pi)$ and, applying \cite[Prop. 1 (d)]{logsmoothcurves}, we deduce the log smoothness of $f_{\log}$.
\end{proof}

Note that Theorem \ref{torsor} implies Theorem \ref{raynaudmodelcriterion} from the introduction. From it, we deduce necessary and sufficient conditions for the existence of a proper, log smooth $T$-model of $Y$.

\begin{corollary}\label{torsorlogsmooth}
Assume $M_T$ is the canonical log structure. Let $A\to T$ be an abelian scheme and $Y$ an $A_u$-torsor. The following are equivalent:
\begin{enumerate}[label=\normalfont(\roman*)]
\item $Y$ can be extended to a surjective, log smooth morphism over $(T,M_T)$
\item $Y$ can be extended to a proper, log smooth morphism over $(T,M_T)$
\item the proper regular $A$-model $X$ of $Y$, with its canonical log structure $M_X$, is log smooth over $(T,M_T)$
\item the proper regular $A$-model of $Y$ is cohomologically flat over $T$.
\end{enumerate}
\end{corollary}
\begin{proof}
By Theorem \ref{torsor}, (iii) $\Leftrightarrow$ (iv) and we trivially have (iii) $\Rightarrow$ (ii) $\Rightarrow$ (i). So it remains to show (i) $\Rightarrow$ (iii). Let $g_{\log}:(Z,M_Z)\to (T,M_T)$ be a surjective log smooth morphism with $Z_u\simeq Y$.

Let $\zeta\in Z_t$ be a maximal point; since $(Z,M_Z)$ is log regular, $\mathcal{O}_{Z,\zeta}$ is a discrete valuation ring. Since $M_T$ is canonical, it follows from \cite[2.6]{niziol} that $\overline{M}_{Z,\zeta}\cong\mathbb{N}$ and $g_{\log}$ is Kummer in an open neighborhood $V$ of $\zeta$. Since $X/T$ is proper, up to further localizing at $\zeta$, we may assume there is a morphism $V\to X$ which is an open immersion on generic fibres over $T$. By \cite[2.6]{niziol}, this morphism naturally extends to a morphism of log schemes.

The fibre product $(X,M_X)\times_{(T,M_T)}(V,M_Z|_V)$ is log smooth over the log regular scheme $(X,M_X)$, hence is itself log regular. In particular, its underlying scheme is normal. On the other hand, base-changing the morphism $\mu\times\pr_X$ by $(V,M_Z|_V)\to (X,M_X)$, we get a morphism
\[ A\times_T(V,M_Z|_V)\to (X,M_X)\times_{(T,M_T)}(V,M_Z|_V) \]
On underlying schemes, this is a finite morphism of normal flat $T$-schemes which is an isomorphism on generic fibres, hence it must be an isomorphism. Now it follows easily from \cite[2.6]{niziol} that it is an isomorphism of log schemes. In particular, $f_{\log}\times_{(T,M_T)}(V,M_Z|_V)$ is log smooth, hence, by log flat descent \cite{logflatdescent}, so is $f_{\log}$.
\end{proof}

Note that this result is a higher-dimensional generalization of a result for genus 1 curves of \cite{mitsuismeets} and \cite{logsmoothcurves} shown under restrictive hypotheses on $T$.

\subsection{Application: curves}\label{curvesection}
Let $T$ be the spectrum of a discrete valuation ring as in \S\ref{conventions} (so $p=\cha(k(t))$), $M_T$ its canonical log structure and $f_{\log}:(X,M_X)\to (T,M_T)$ a morphism of fs log schemes with underlying morphism of schemes $f:X\to T$ that we assume to be flat. Denote by $l$ a prime number invertible on $T$.

We can improve the main result of \cite{logsmoothcurves} by removing the perfectness assumption on $k(t)$ and allowing the log structure to contain some horizontal components (cf. \cite{saito2,stix}).

\begin{theorem}\label{curves}
Assume $X$ is regular of dimension 2, the generic fibre of $f_{\log}$ is log smooth. Let $g:U\hookrightarrow X$ be the largest open subset on which $M_X$ is trivial. If $f$ is proper, then $f_{\log}$ is log smooth if and only if
\begin{enumerate}[label=\normalfont(\roman*)]
\item $M_X=g_*\mathcal{O}_U^*\cap\mathcal{O}_X$
\item $(X\setminus U)\subset X$ is a normal crossings divisor
\item $H^i_{\et}(U_{\bar{u}}
,\mathbb{Q}_l)$ is tamely ramified for $i\leq 1$ (where $\bar{u}\to u$ is a geometric point)
\item $f$ is cohomologically flat
\item $X_0:=(X_t)_{\red}$ is normal crossings scheme over $t$ (Def. \ref{ncdscheme}), such that any component of $X_t$ of multiplicity divisible by $p$ has no self-intersection and no two such components intersect.
\end{enumerate}
\end{theorem}

The next result is similar to \cite[Prop. 1]{logsmoothcurves}, but without restriction on $k(t)$.\footnote{We point out an inaccuracy in \cite{logsmoothcurves}: the first statement of part (e) of \cite[Prop. 1]{logsmoothcurves} is (rather trivially) false but holds if one assumes $f$ to be log smooth (which is what was needed in that paper)---see part (vi) of Proposition \ref{ncd}.}

\begin{proposition}\label{ncd}
Assume $X$ is regular and the generic fibre of $f_{\log}$ is log smooth. Let $j:X_u\to X$ be the inclusion of the generic fibre of $f$ and define a log structure $N_X:=j_*\mathcal{O}_{X_u}^*\cap\mathcal{O}_X$.
\begin{enumerate}[label=\normalfont(\roman*)]
\item  $(X,M_X)$ is log regular if and only if
\begin{enumerate}[label=\normalfont(\arabic*)]
\item $M_X=g_*\mathcal{O}_U^*\cap\mathcal{O}_X$, where $g:U\hookrightarrow X$ is the largest open subset on which $M_X$ is trivial
\item $X\setminus U$ is a normal crossings divisor on $X$.
\end{enumerate}
\end{enumerate}
Moreover, if $(X,M_X)$ is log regular at a geometric point $x\to X_t$, then
\begin{enumerate}[resume,label=\normalfont(\roman*)]
\item $\overline{M}_{X,x}\simeq\mathbb{N}^r$, and preimages $x_1,...,x_r\in \mathcal{O}_{X,x}$ of a basis of $\overline{M}_{X,x}$ form part of a regular system of parameters
\item there is a canonical map $N_{X,x}\to M_{X,x}$; in particular, a uniformizer $\pi\in\mathcal{O}(T)$ can be written $\pi=v\prod_{i=1}^rx_i^{a_i}$ with $v\in\mathcal{O}_{X,x}^*$ and $a_i\in\mathbb{N}$.
\end{enumerate}
Furthermore, with the notation of \emph{(ii)--(iii)}, we have
\begin{enumerate}[resume,label=\normalfont(\roman*)]
\item if $N_{X,x}\neq M_{X,x}$, then $a_i=0$ for some $1\leq i\leq r$
\item if $p\nmid a_{i}$ for some $1\leq i\leq r$ and $k(t)\subset k(x)$ is separable, then $f_{\log}$ is log smooth at $x$
\item if $f_{\log}$ is log smooth and $p\mid a_i$ for all $1\leq i\leq r$, then $r<\dim\mathcal{O}_{X,x}$
\item if $f_{\log}$ is log smooth, then locally at $x$ there is a smooth morphism
\[ X\to\Spec(\mathcal{O}_{T}[y,x_1,...,x_r]/(\pi-y\prod_{i=1}^{r}x_i^{a_i})) \]
\item if $f_{\log}$ is log smooth, then $X_0$ is a normal crossings scheme over $t$.
\end{enumerate}
\end{proposition}
\begin{proof}
(i) is local at a geometric point $x\to X$. First assume $(X,M_X)$ is log regular. Then $M_X=g_*\mathcal{O}_U^*\cap\mathcal{O}_X$ by \cite[2.6]{niziol}. By \cite[Prop. 3.2]{logreg} there is a sharp fs monoid $P$ of rank $r\leq d:=\dim\mathcal{O}_{X,x}$ inducing the log structure $M_X$, a regular local ring $R$ of dimension at most 1, and elements $x_{r+1},...,x_d$ forming part of a regular system of parameters such that $\widehat{\mathcal{O}}_{X,x}=R[[P]][[x_{r+1},...,x_d]]/(\theta)$, where $\theta$ has constant term equal to $\cha(k(x))$. Using the regularity of $X$, one easily shows there are $x_1,...,x_r\in P$ such that $P=\oplus_{i=1}^r\mathbb{N}x_i$ and $x_1,...,x_d$ form a regular sequence of parameters (cf. \cite[5.2]{niziol} or \cite[proof of Prop. 1]{logsmoothcurves}). Then $\prod_{i=1}^rx_i$ defines a normal crossings divisor whose support is equal to $X\setminus U$. This implies the necessity of the conditions.

For the sufficiency, let $x_1,...,x_d\in\mathcal{O}_{X,x}$ be a regular system of parameters such that $\prod_{i=1}^rx_i$ defines the normal crossings divisor $X\setminus U$. The elements $x_1,...,x_r\in g_*\mathcal{O}_U^*\cap\mathcal{O}_X=M_X$ correspond to the irreducible components of $X\setminus U$ and we have $g_*\mathcal{O}_U^*/\mathcal{O}_X^*=\oplus_{i=1}^r\mathbb{Z}x_i$. Thus, the monoid $\oplus_{i=1}^r\mathbb{N}x_i$ induces $M_X$ at $x$, and $(X,M_X)$ is log regular at $x$ by definition, proving (i) as well as (ii).

Since $M_T$ is the canonical log structure, we have $U\subset X_u$, hence $N_X\subset M_X$ (by (i)), proving the first assertion of (iii); the second follows immediately. Note that the prime divisors of $\pi$ are the $x_i$ for which $a_i>0$, so if $N_{X,x}\neq M_{X,x}$ then $a_i=0$ for some $1\leq i\leq r$, hence (iv).

For (v), taking an $a_i$th root of $v$ we have $\pi=(v^{1/a_{i}}x_i)^{a_i}\prod_{1\leq j\leq r, j\neq i}x_j^{a_j}$ from which we can easily construct a log smooth chart as in \cite[3.5]{log}. The details are left to the reader.

For the proof of (vi)--(viii), assume $f_{\log}$ is log smooth and $X$ connected of dimension $d=\dim\mathcal{O}_{X,x}$. Consider the two cases
\begin{enumerate}[label=\normalfont(\alph*)]
\item $p\mid a_i$ for all $1\leq i\leq r$
\item $p\nmid a_i$ for some $1\leq i\leq r$.
\end{enumerate}
In case (b), we may assume $v^{1/a_i}\in\mathcal{O}_X$, hence, up to relabeling we may assume $v=1$ in this case.

For a $T$-scheme $S$ with log structure, write $\omega^1_{S}$ for the logarithmic differentials over $T$ with the trivial log structure. If $S$ is a log scheme over $(T,M_T)$, then write $\omega^1_{S/T}$ for the relative log differentials over $(T,M_T)$, and similarly for morphisms of log schemes. Since $f_{\log}$ is smooth, by \cite[3.12]{log} we have an exact sequence
\[ 0\to f^*\omega^1_{T}\to\omega^1_X\to\omega^1_{X/T}\to 0 \]
with $\omega^1_{X/T}$ a vector bundle of rank $d-1$. Pulling back to $X_t$, we deduce an exact sequence
\[ 0\to f^*\omega^1_t\to\omega^1_{X_t}\to\omega^1_{X_t/t}\to 0 \]
Since $\omega^1_T=k(t)\dlog(\pi)$, it follows that $\omega^1_{X_t}$ is locally free of rank $d$ with $\dlog(\pi)$ forming part of a basis.

Consider the closed subscheme $Y\subset X$ of ideal $I=(x_1,...,x_r)$. There is an $\mathcal{O}_Y$-linear residue map $\rho:\omega^1_{X}|_Y\to\mathcal{O}_Y^r$ mapping $\dlog(x_1),...\dlog(x_r)$ to a basis of the target, whose kernel contains the usual differentials $\Omega^1_{X/T}|_Y$ (\cite[Prop. 1]{logsmoothcurves}, cf. \cite[IV, 2.3.5]{logbook}). So $\dlog(x_1),...\dlog(x_r)$ form part of a basis of $\omega^1_X|_Y$. In case (a) we have $\dlog(\pi)\equiv\dlog(v)\mod p$, hence $\rho(\dlog(\pi))=0$. It follows that $\text{d}v,\dlog(x_1),...,\dlog(x_r)$ form part of a basis of $\omega^1_X|_Y$ in this case.

Endowing $Y$ with the inverse image log structure of $X$, by \cite[IV, 2.3.2]{logbook} we have an exact sequence of log differentials
\[ I/I^2\to\omega^1_X|_Y\to\omega^1_Y\to 0 \]
and clearly the left hand map is zero. Hence $\omega^1_Y\cong\omega^1_X|_Y$.

Let $y$ be an indeterminate in case (a) (resp. $y=1$ in case (b)). Define a log scheme $(U,M_U)$ by
\[ U=\Spec(\mathcal{O}(T)[y^{\pm 1},x_1,...,x_r]/(\pi-y\prod_{i=1}^rx_i^{a_i})) \]
and $M_U$ the natural log structure induced by $\oplus_{i=1}^r\mathbb{N}x_i$. Let $V\subset U$ be the closed subscheme of ideal $(x_1,...,x_r)$. The natural morphism $X\to U$ mapping $y$ to $v$ induces strict morphisms of log schemes $(X,M_X)\to (U,M_U)$ and  $(Y,M_X|_Y)\to (V,M_U|_V)$. Denoting the underlying morphism $h:Y\to V$, we have a right-exact sequence
\begin{equation}\label{Yexact}
0\to h^*\Omega^1_{V/t}\to\Omega^1_{Y/t}\to\Omega^1_{Y/V}\to 0
\end{equation}
Since $\omega^1_{V}\cong\omega^1_{U}|_V$, $\omega^1_{V}$ is a free $\mathcal{O}_V$-module generated by $\text{d}y,\dlog(x_1),...,\dlog(x_{r})$. It follows that $\omega^1_{Y/V}\cong\Omega^1_{Y/V}$ is locally free of rank $d-(r+1)$ in case (a) (resp. $d-r$ in case (b)). Since $\text{d}v$ vanishes nowhere in case (a) (resp. $V=t$ in case (b)) and $Y$ is regular, hence reduced, the sequence (\ref{Yexact}) must also be left exact. Hence, $\Omega^1_{Y/t}$ is locally free of rank $d-r$. Since $\dim Y=d-r$, by \cite[IV\textsubscript{4}, 17.15.5]{ega} $Y$ is smooth over $t$. It now follows from (\ref{Yexact}) that $h$ is smooth. Applying \cite[0\textsubscript{IV}, 15.1.21]{ega}, we see that $X\to U$ is flat along $Y$, hence smooth, proving (vii) as well as (vi). Finally, (viii) follows from (vii).
\end{proof}

For the proof of Theorem \ref{curves} we will need a couple of lemmata.

\begin{lemma}\label{steintame}
Assume the generic fibre of $f_{\log}$ is log smooth and $f$ is proper and satisfies (N). Let $X\to T'\to T$ be the Stein factorization of $f$. If $\bar{u}\to T$ is a geometric point lying over $u$ and $H^0_{\et}(X_{\bar{u}},\mathbb{Z}/n)$ is tamely ramified for some integer $n>1$ prime to $p$, then $T'\to T$ is a tamely ramified covering. In particular, this holds if $f_{\log}$ is log smooth.
\end{lemma}
\begin{proof}
By assumption, the geometric generic fibre of $f_{\log}$ is log regular, hence reduced. This implies that $T'_u$ is the spectrum of a product of separable extensions of $k(u)$ (\cite[IV\textsubscript{2}, 4.6.1]{ega}).

We have $H^0_{\et}(T'_{\bar{u}},\mathbb{Z}/n)=H^0_{\et}(X_{\bar{u}},\mathbb{Z}/n)$, and $T'$ is normal (Lemma \ref{steinreduction}). Now apply the next lemma to deduce the first statement. For the second statement, note that $X$ is normal in this case and, by Nakayama's theorem \cite{nakayama}, the tame ramification condition holds.
\end{proof}

\begin{lemma}\label{tame}
Let $k(u)\subset K$ be a finite separable extension and $S=\Spec(K)$. If $H^0_{\et}(S\times_u\bar{u},\mathbb{Z}/n)$ is tamely ramified for some integer $n>1$, then $k(u)\subset K$ is a tamely ramified extension. The same holds for $\mathbb{Z}_l$ or $\mathbb{Q}_l$ coefficients.
\end{lemma}
\begin{proof}
We may assume $k(u)\subset k(\bar{u})$ is separable. Let $P\subset\Gal(k(\bar{u})/k(u))$ be the wild inertia group. Then, $H^0_{\et}(S_{\bar{u}},\mathbb{Z}/n)^P=(\mathbb{Z}/n)^e$, where $e$ is the number of connected components of $S_{\bar{u}}/P=\Spec(K\otimes_{k(u)}k(\bar{u})^P)$. If $H^0_{\et}(S_{\bar{u}},\mathbb{Z}/n)$ is tamely ramified, i.e., $P$ acts trivially on $H^0_{\et}(S_{\bar{u}},\mathbb{Z}/n)$, then any connected component of $S_{\bar{u}}/P$ is geometrically connected over $k(\bar{u})^P$, hence must be isomorphic to $\Spec(k(\bar{u})^P)$, whence $K\subset k(\bar{u})^P$. This proves the first statement and the second can be reduced to this one.
\end{proof}

\begin{proof}[Proof of Theorem \ref{curves}]
We first note that the conditions are necessary: (i) follows from \cite[2.6]{niziol}, (ii) from Proposition \ref{ncd}, (iii) from Nakayama's theorem \cite{nakayama}, (iv) from Theorem \ref{main}, and (v) follows from (ii) and Proposition \ref{ncd}.

To show the sufficiency, first note that since $X_u$ is log smooth over $u$ it is geometrically normal, hence smooth. In fact, locally for the \'etale topology its log structure can be given by the monoid $\mathbb{N}$ and an \'etale map $k(u)[\mathbb{N}]\to\mathcal{O}_{X_u}$. This implies that the log structure of $X_u$ is given by a finite set of closed points whose residue fields are separable extensions of $k(u)$. Moreover, (iii) implies that these extensions are tamely ramified: indeed, we have an exact sequence
\[ H^1_{\et}(U_{\bar{u}},\mathbb{Q}_l)\to H^0_{\et}((X\setminus U)_{\bar{u}},\mathbb{Q}_l)(-1)\to H^2_{\et}(X_{\bar{u}},\mathbb{Q}_l) \]
and $H^2_{\et}(X_{\bar{u}},\mathbb{Q}_l)=\oplus\mathbb{Q}_l(-1)$ is unramified, so, since the wild inertia group is pro-$p$ (\cite[IV, \S 2]{localfields}), (iii) implies $H^0_{\et}((X\setminus U)_{\bar{u}},\mathbb{Q}_l)$ is tamely ramified and the claim follows from Lemma \ref{tame}.

Let $N_X$ be as in the statement of Proposition \ref{ncd}. If $x\to X_t$ is a geometric point, then $\overline{N}_{X,x}\neq 0$. If $\overline{M}_{X,x}^{\gp}\neq \overline{N}_{X,x}^{\gp}$, then by Proposition \ref{ncd} we have $\overline{M}_{X,x}\cong\mathbb{N}^2$ and there are $x_1,x_2\in M_{X,x}$ mapping to a basis of $\overline{M}_{X,x}$ such that $\pi=vx_2^a$ for some $v\in\mathcal{O}_{X,x}^*$. The closed subscheme of $T'\subset X_{(x)}$ of equation $x_1=0$ is regular, flat over $T$, and $T'_u$ is a point of $X_u\setminus U$, hence $T'\to T$ is tamely ramified by the above. In particular, this implies that $k(x)$ is a separable extension of $k(t)$. Since $x_1,x_2$ generate the maximal ideal of $\mathcal{O}_{X,x}$, the image of $x_2$ in $\mathcal{O}_{T'}$ is a uniformizer, hence $p\nmid a$. By Proposition \ref{ncd} (v), $f_{\log}$ is log smooth at this point.

Now suppose $x\to X_t$ is a geometric point such that $\overline{M}_{X,x}^{\gp}=\overline{N}_{X,x}^{\gp}$. In this case we have $(g_*\mathcal{O}_U^*)_x=(j_*\mathcal{O}_{X_u}^*)_x$, hence $M_{X,x}=N_{X,x}$ by (i). We reduce this case to the case $k(t)$ is algebraically closed as follows. First of all, we may assume $T$ to be strictly local. Moreover, by Lemma \ref{steintame} and Lemma \ref{steinreduction} we may assume $f_*\mathcal{O}_X=\mathcal{O}_T$, hence the fibres of $f$ are geometrically connected. Let $S$ be the spectrum of a henselian discrete valuation ring with algebraically closed residue field as in Lemma \ref{acreduction} and let $\bar{v}\to S_u$ be a geometric point lying above $\bar{u}$. Then, $f\times_TS$ is cohomologically flat and, by Lemma \ref{semistablelemma}, $X_S$ is regular and $X_0\times_TS\subset X\times_TS$ is a normal crossings divisor. Note that since $X_0$ is a normal crossings scheme, its multiplicities are equal to its geometric multiplicities. Hence, (v) holds for $f\times_TS$. Moreover, the action of the wild inertia subgroup of $\pi_1(S_u,\bar{v})$ on $H^i_{\et}(U_{\bar{v}},\mathbb{Q}_l)\cong H^i_{\et}(U_{\bar{u}},\mathbb{Q}_l)$ factors through that of the wild inertia subgroup of $\pi_1(u,\bar{u})$, hence (iii) holds for $X_S$. Applying \cite[Thm. 1]{logsmoothcurves}, we deduce that the induced morphism $(X_S,N_{X_S})\to (S,M_S)$ is log smooth. Thus, $f_{\log}\times_TS$ is log smooth above $x$ and so is $f_{\log}$ by Lemma \ref{acreduction}.
\end{proof}

\bibliographystyle{amsplain}
\bibliography{gabber.bib}

@article{nakayamalogetale,
 author = {C. Nakayama},
 title = {Logarithmic {\'e}tale cohomology},
 fjournal = {Mathematische Annalen},
 journal = {Math. Ann.},
 volume = {308},
 number = {3},
 pages = {365--404},
 year = {1997},
}

@Article{logII,
 Author = {K. Kato},
 Title = {Logarithmic structures of {F}ontaine--{I}llusie. {II}: {L}ogarithmic flat topology},
 FJournal = {Tokyo Journal of Mathematics},
 Journal = {Tokyo J. Math.},
 Volume = {44},
 Number = {1},
 Pages = {125--155},
 Year = {2021},
}

@Article{illusie-kato-nakayama,
 Author = {L. Illusie and K. Kato and C. Nakayama},
 Title = {Quasi-unipotent logarithmic {R}iemann--{H}ilbert correspondences},
 FJournal = {Journal of Mathematical Sciences. University of Tokyo},
 Journal = {J. Math. Sci., Tokyo},
 Volume = {12},
 Number = {1},
 Pages = {1--66},
 Year = {2005},
}

@Article{GLL,
 Author = {O. Gabber and Q. Liu and D. Lorenzini},
 Title = {The index of an algebraic variety},
 FJournal = {Inventiones Mathematicae},
 Journal = {Invent. Math.},
 Volume = {192},
 Number = {3},
 Pages = {567--626},
 Year = {2013},
}

@Article{ferrand,
 Author = {D. Ferrand},
 Title = {Conducteur, descente et pincement},
 FJournal = {Bulletin de la Soci{\'e}t{\'e} Math{\'e}matique de France},
 Journal = {Bull. Soc. Math. Fr.},
 Volume = {131},
 Number = {4},
 Pages = {553--585},
 Year = {2003},
}

@article {lewin-menegaux,
    AUTHOR = {R. Lewin-M\'{e}n\'{e}gaux},
     TITLE = {Mod\`eles minimaux de torseurs},
   JOURNAL = {C. R. Acad. Sci. Paris S\'{e}r. I Math.},
    VOLUME = {297},
      YEAR = {1983},
    NUMBER = {4},
     PAGES = {257--260},
}

@incollection{brauer2,
	author = {A. Grothendieck},
	booktitle = {Dix expos\'{e}s sur la cohomologie des sch\'{e}mas},
	publisher = {North-Holland Pub. Co.},
	title = {Le groupe de {Brauer} {II}. {T}h\'eorie cohomologique },
	year = {1968}}

@inproceedings{raynaudquotient,
	author = {M. Raynaud},
	booktitle = {Proceedings of a Conference on Local Fields},
	editor = {Springer, T. A.},
	pages = {78--85},
	publisher = {Springer},
	title = {Passage au quotient par une relation d'{\'e}quivalence plate},
	year = {1967}}

@Incollection{katzST,
 Author = {N. Katz},
 Editor = {J. Giraud and L. Illusie and M. Raynaud},
 Title = {Serre-{Tate} local moduli},
 Year = {1981},
 Booktitle = {Surfaces {A}lg\'ebriques, {S}\'eminaire de G\'eom\'etrie {A}lg\'ebrique d'{Orsay} 1976-78},
 Series = {Lecture Notes in Mathematics},
 volume = {868},
 pages = {138-202},
 publisher = {Springer},
}

@Book{concretemath,
 Author = {R. L. Graham and D. E. Knuth and O. Patashnik},
 Title = {Concrete Mathematics},
 Edition = {2nd ed.},
 Year = {1994},
 Publisher = {Addison-Wesley},
}

@Article{logflatdescent,
 Author = {L. Illusie and C. Nakayama and T. Tsuji},
 Title = {On log flat descent},
 Journal = {Proc. Japan Acad., Ser. A},
 Volume = {89},
 Number = {1},
 Pages = {1--5},
 Year = {2013},
}

@Article{kollar,
 Author = {J. Koll{\'a}r},
 Title = {Variants of normality for {N}oetherian schemes},
 Journal = {Pure Appl. Math. Q.},
 Volume = {12},
 Number = {1},
 Pages = {1--31},
 Year = {2016},
}

@InCollection{illusie,
 Author = {L. Illusie},
 Title = {An overview of the work of {K}. {Fujiwara}, {K}. {Kato}, and {C}. {Nakayama} on logarithmic {\'e}tale cohomology},
 BookTitle = {Cohomologies \(p\)-adiques et applications arithm\'etiques (II)},
 Pages = {271--322},
 Year = {2002},
 Publisher = {Soci{\'e}t{\'e} Math{\'e}matique de France},
}

@Book{sga3.1,
    editor = {M. Demazure and A. Grothendieck},
	title = {Sch\'emas en Groupes ({SGA} 3, tome {I})},
    Series = {Documents Math{\'e}matiques},
    Note = {new annotated edition of the 1970 original published by {Springer}},
     Volume = {7},
     Year = {2011},
    Publisher = {Soci{\'e}t{\'e} Math{\'e}matique de France},
}

@article{stix,
	author = {J. Stix},
	journal = {J. Algebraic Geom.},
	pages = {119-136},
	title = {A logarithmic view towards semistable reduction},
	volume = {14},
	year = {2005}}

@article{saito2,
	author = {T. Saito},
	journal = {J. Algebraic Geom.},
	pages = {287-321},
	title = {Log smooth extension of a family of curves and semi-stable reduction},
	volume = {13},
	year = {2004}}

@article{LLR,
	author = {Q. Liu and D. Lorenzini and M. Raynaud},
	journal = {Invent. Math.},
	pages = {455-518},
	title = {N\'eron models, {L}ie algebras, and reduction of curves of genus one},
	volume = {157},
	year = {2004}}

@unpublished{mitsuismeets,
	author = {K. Mitsui and A. Smeets},
	note = {arXiv preprint},
	title = {Logarithmic good reduction and the index},
	year = {2017}}

@article{tsuji,
	author = {T. Tsuji},
	journal = {Tunis. J. Math.},
	number = {2},
	pages = {185--220},
	title = {Saturated morphisms of logarithmic schemes},
	volume = {1},
	year = {2019}}

@article{logsmoothcurves,
	author = {R. Lodh},
	journal = {Manuscr. Math.},
	number = {1-2},
	pages = {197--211},
	title = {Log smooth curves over discrete valuation rings},
	volume = {167},
	year = {2022}}

@incollection{picard,
	author = {S. L. Kleiman},
	booktitle = {Fundamental Algebraic Geometry},
	publisher = {American Mathematical Society},
	series = {Math. Surv. Monogr.},
	title = {The {P}icard scheme},
	volume = {123},
	year = {2005}}

@incollection{grothpicard,
	author = {A. Grothendieck},
	booktitle = {S\'eminaire Bourbaki},
	publisher = {Soci\'et\'e math\'ematique de France},
	title = {Technique de descente et th\'eor\`emes d'existence en g\'eom\'etrie alg\'ebrique. {I--VI}},
	year = {1959--1962}}

@incollection{artinformalI,
	author = {M. Artin},
	booktitle = {Global {A}nalysis, {P}apers in {H}onor of {K}. {K}odaira},
	pages = {21-71},
	publisher = {Princeton University Press},
	title = {Algebraization of formal moduli: {I}},
	year = {1969}}

@article{artinversal,
	author = {M. Artin},
	journal = {Invent. Math.},
	pages = {165--189},
	title = {Versal deformations and algebraic stacks},
	volume = {27},
	year = {1974}}

@book{neronmodels,
	author = {S. Bosch and W. L\"utkebohmert and M. Raynaud},
	publisher = {Springer},
	series = {{Ergebnisse der Mathematik und ihrer Grenzgebiete. 3. Folge}},
	title = {{N\'eron {M}odels}},
	volume = {21},
	year = {1990}}

@article{raynaud,
	author = {M. Raynaud},
	journal = {Publ. Math. I.H.E.S.},
	pages = {27--76},
	title = {Sp\'ecialisation du foncteur de {P}icard},
	volume = {38},
	year = {1970}}

@book{localfields,
	author = {J-P. Serre},
	publisher = {Springer},
	series = {Graduate Texts in Mathematics},
	title = {Local Fields},
	volume = {67},
	year = {1979}}

@article{niziol,
	author = {W. Nizio\l},
	journal = {J. Algebraic Geom.},
	number = {1},
	pages = {1-29},
	title = {Toric singularities: log-blow-ups and global resolutions},
	volume = {15},
	year = {2006}}

@book{sga4.3,
	author = {M. Artin and A. Grothendieck and J.-L. Verdier},
	publisher = {Springer},
	series = {Lecture Notes in Mathematics},
	title = {Th\'eorie des topos et cohomologie \'etale des sch\'emas ({SGA} 4), tome 3},
	volume = {305},
	year = {1973}}

@article{logreg,
	author = {K. Kato},
	journal = {{American J. Math.}},
	pages = {1073-1099},
	title = {Toric singularities},
	volume = {116},
	year = {1994}}

@article{ega,
	author = {A. Grothendieck},
	journal = {Publ. Math. I.H.E.S.},
	note = {written in collaboration with J. Dieudonn\'e},
	title = {\'{E}l\'ements de g\'eom\'etrie alg\'ebrique},
	volume = {4, 8, 11, 17, 20, 24, 28, 32},
	year = {1960-1967}}

@inproceedings{log,
	author = {K. Kato},
	booktitle = {Algebraic Analysis, Geometry and Number Theory},
	editor = {J.-I. Igusa},
	pages = {191-224},
	publisher = {Johns Hopkins University Press},
	title = {Logarithmic structures of {F}ontaine-{I}llusie},
	year = {1989}}

@book{logbook,
	author = {A. {Ogus}},
	publisher = {Cambridge University Press},
	series = {{Cambridge Studies in Advanced Mathematics}},
	title = {{Lectures on Logarithmic Algebraic Geometry}},
	volume = {178},
	year = {2018}}

@article{nakayama,
	author = {C. Nakayama},
	journal = {Compositio Math.},
	pages = {45-75},
	title = {Nearby cycles for log smooth families},
	volume = {112},
	year = {1998}}

@book{ega1,
	author = {A. Grothendieck and J. A. Dieudonn\'e},
	publisher = {Springer},
	series = {Grundlehren der mathematischen Wissenschaften},
	title = {{\'E}l\'ements de G\'eom\'etrie Alg\'ebrique {I}},
	volume = {166},
	year = {1971},
}

@book{mumfordAV,
	author = {D. Mumford},
	edition = {2nd},
	publisher = {Oxford University Press},
	title = {Abelian Varieties},
	year = {1974}}

\end{document}